\documentclass[10pt,english]{amsart}
\usepackage[T1]{fontenc}
\usepackage[charter]{mathdesign}
\usepackage{hyperref}
\usepackage[paper = letterpaper, 
left = 3cm, 
right = 3cm, 
top = 3cm, 
bottom = 2.95cm]{geometry}
\newtheorem{theorem}{Theorem}[section]
\newtheorem*{theorem1}{Theorem \ref{roughclassi}}
\newtheorem*{theorem2}{Theorem \ref{thm mukai times invol}}
\newtheorem{lemma}[theorem]{Lemma}
\newtheorem{proposition}[theorem]{Proposition}

\theoremstyle{definition}
\newtheorem{remark}[theorem]{Remark}
\newtheorem{example}[theorem]{Example}
\newtheorem{definition}[theorem]{Definition}

\title[K3-surfaces with involution and maximal symplectic symmetry]{Classification of K3-surfaces with involution\\ 
and maximal symplectic symmetry}
\author[Kristina Frantzen]{Kristina Frantzen}
\address{Kristina Frantzen\\Mathematisches Institut\\Universit\"at Bayreuth\\Germany\\
kristina.frantzen@uni-bayreuth.de}
\date {June 10, 2009}
\thanks{Research supported by Studienstiftung des
    deutschen Volkes and Deutsche Forschungsgemeinschaft}
\begin{document}
\begin{abstract}
K3-surfaces with antisymplectic involution and compatible symplectic actions of finite groups are considered. 
In this situation actions of large finite groups of symplectic transformations are shown to arise via double covers of Del Pezzo surfaces. A complete classification of K3-surfaces with maximal symplectic symmetry is obtained.
\end{abstract}
\maketitle
%
%
%
%
\vspace*{-10mm}
\section{Introduction}
An
\emph{antisymplectic involution} $\sigma$ on a complex K3-surface $X$ is a holomorphic involution acting nontrivially on the space $H^0(X,\Omega^2) = \mathbb C \cdot \omega $ of holomorphic 2-forms on $X$, i.e., $\sigma^* \omega = -\omega$.
Here we study K3-surfaces with antisymplectic involution from the point of view of symmetry and
consider actions of finite groups $G$ of symplectic transformations which are compatible with $\sigma$ in the sense that every $g \in G$ is a holomorphic automorphism of $X$ with $g^* \omega = \omega$ and $g \sigma = \sigma g$.
\subsection*{Main results}
Given a finite group $G$ we investigate if it can act in a compatible fashion on a K3-surface $X$ with antisymplectic involution
$\sigma$. If this is the case, then already the order of $G$ yields strong constraints on the geometry of $X$ and we obtain the following description.
\begin{theorem1}
Let $X$ be a K3-surface with a symplectic action of a finite group $G$ centralized by an
antisymplectic involution $\sigma$ and assume that
$\mathrm{Fix}_X(\sigma)\neq \emptyset$. If $|G|>96$, then $X/\sigma$ is a ($G$-minimal)
Del Pezzo surface and $\mathrm{Fix}_X(\sigma)$ is a smooth connected curve $B$ with $ g(B)\geq
3$.
\end{theorem1}
By a theorem due to Mukai \cite{mukai} finite groups of symplectic
transformations on K3-surfaces are characterized by the existence of a
certain embedding into a particular Mathieu group and are 
subgroups of the following eleven finite groups of maximal symplectic symmetry (cf.\ also \cite{kondo}).
\begin{table}[h]
\renewcommand{\arraystretch}{1.3}
\centering
\begin{tabular}{c|c|c|c|c|c|c|c|c|c|c|c} 
   & 1 & 2 & 3 & 4 & 5 & 6 & 7 & 8 & 9 & 10 & 11  \\ \hline 
 group & $L_2(7)$& $A_6$& $S_5$& $M_{20}$& $F_{384}$& $A_{4,4}$& $T_{192}$& $H_{192}$ &$N_{72}$& $M_9$& $ T_{48}$ \\ \hline 
order & 168 & 360 & 120 & 960 & 384 & 288 & 192 & 192 & 72 &72& 48
\end{tabular}
\vspace{3mm}
\caption{Maximal finite groups of symplectic automorphisms on K3-surfaces}\label{TableMukai}
\end{table}
\vspace{-5mm}\\
\renewcommand{\arraystretch}{1}
This result naturally limits 
the possible choices of $G$ above
and has
led us to consider the classification problem for groups $G$ from this list of eleven groups:
as a refinement of Theorem \ref{roughclassi} we obtain the following complete classification.
\begin{theorem2}
Let $G$ be a group from Table \ref{TableMukai} acting on a K3-surface $X$ by symplectic transformations and let $\sigma $ be an antisymplectic involution on $X$ centralizing $G$ with $\mathrm{Fix}_X(\sigma) \neq \emptyset$. Then the pair $(X,G)$ is equivariantly isomorphic to a surface in Table \ref{Mukai times invol} below. In particular, for the groups $G$ numbered 4-8 on Mukai's list, there does not exist a K3-surface with an action of $G \times C_2$ satisfying the properties above.
\begin{table}[h]
\renewcommand{\arraystretch}{1.3}
\centering
\begin{tabular}{l|l|l}
  & $G$  & \textbf{K3-surface} $X$ \\ \hline 
1a & $L_2(7)$  & $\{x_1^3x_2+x_2^3x_3+x_3^3x_1+x_4^4 =0\} \subset \mathbb P_3$\\ \hline
1b & $L_2(7)$  & Double cover of $\mathbb P_2$ branched along 
		$\{x_1^5x_2+x_3^5x_1+x_2^5x_3-5x_1^2 x_2^2 x_3^2 =0\}$\\ \hline
2 & $A_6$  & Double cover of $\mathbb P_2$ branched along \\
& &  		$\{10 x_1^3x_2^3+ 9 x_1^5x_3 + 9 x_2^3x_3^3-45 x_1^2 x_2^2 x_3^2-135 x_1 x_2 x_3^4 + 27 			x_3^6  =0\}$\\ \hline
3a & $S_5$  & $\{\sum_{i=1}^5 x_i = \sum_{i=1}^6 x_1^2 = \sum_{i=1}^5 x_i^3=0\} \subset \mathbb P_5$\\ \hline
3b & $S_5$  & Minimal desingularization of the double cover of $\mathbb P_2$ branched along \\
 & &  $\{2x_1^4x_2x_3+2x_1x_2^4x_3+2x_1x_2x_3^4-2x_1^4x_2^2-2x_1^4x_3^2-2x_1^2x_2^4-2x_1^2x_3^4$\\
 & & $-2x_2^4x_3^2-2x_2^2x_3^4+2x_1^3x_2^3+2x_1^3x_3^3+2x_2^3x_3^3+x_1^3x_2^2x_3+x_1^3x_2x_3^2$\\
 & &  $+x_1^2x_2^3x_3+x_1^2x_2x_3^3+x_1x_2^3x_3^2 + x_1x_2^2x_3^3-6x_1^2x_2^2x_3^2=0\}$ \\ \hline
9 & $N_{72}$ & $\{ x_1^3+ x_2 ^3 + x_3^3 +x_4^3= x_1x_2 + x_3x_4+ x_5^2 = 0 \} \subset \mathbb P_4$ \\ \hline
10 & $M_9$ & Double cover of $\mathbb P_2$ branched along $\{x_1^6+x_2^6 +x_3^6 -10(x_1^3x_2^3 + x_2^3x_3^3 +x_3^3x_1^3) =0\}$\\ \hline
11a & $ T_{48}$  & Double cover of $\mathbb P_2$ branched along $\{x_1x_2(x_1^4-x_2^4)+ x_3^6 =0\}$\\ \hline
11b & $ T_{48}$  & Double cover of $\{  x_1x_2(x_1^4-x_2^4)+ x_3^3+x_4^2=0 \} \subset \mathbb P(1,1,2,3)$\\
&  		& branched along $\{x_3=0\} $
\end{tabular}
\vspace{2mm}
\caption{K3-surfaces with $G \times C_2$-symmetry}
\label{Mukai times invol}
\end{table}
\renewcommand{\arraystretch}{1}
\end{theorem2}
In addition to Examples 1a, 3a, 9, 10, and 11a, which have already been described by Mukai (Example 0.4 in \cite{mukai}) the table contains additional examples (1b, 2, 3b, and 11b) of K3-surfaces with maximal
symplectic symmetry. These arise as equivariant double covers of Del Pezzo surfaces.
\begin{remark}
Besides the above mentioned classification of finite groups of symplectic transformations on K3-surfaces by Mukai (cf.\ also Nikulin's classification in the Abelian case \cite{NikulinFinite} and Kond\=o's alternative approach using lattice theory \cite{kondo}) one of our starting points has been the study of K3-surfaces with $L_2(7)$-symmetry in \cite{OZ168}. Although our approach here is independent of the results by Keum, Oguiso, and Zhang (cf.\ e.g.\ \cite{OZ168}, \cite{KOZ1}, \cite{KOZ2}, and the summary \cite{ZhangSummary}), our understanding of finite group actions on K3-surfaces strongly gained from the numerous contributions to the subject by the authors above.
\end{remark}
We conclude this introduction with an outline of our general classification strategy.
\subsection*{Equivariant minimal model program}
The quotient of a K3-surface by an antisymplectic involution $\sigma$ with fixed points centralized by a finite group $G$ is a rational $G$-surface $Y$ to which we apply an equivariant version of the minimal model program respecting finite symmetry groups (cf.\ Example 2.18 in \cite{kollarmori} and Section 2.3  in \cite{Mori}). A detailed exposition using the language of Mori theory and including in particular equivariant analogues of the cone and contraction
theorems is given e.g.\ in \cite{kf}, Chapter 2 and \cite{FKH2}, Section 4. Here we briefly summarize the program in the following classification
result. 
\begin{proposition}\label{Gmmp}
Let $Y$ be a smooth projective surface with an action of a finite
group $G$.
Then there exists a sequence of $G$-equivariant
extremal contractions $Y \to Y_{(1)} \to \dots \to Y_{\mathrm{min}}$ such that
$Y_{\mathrm{min}}$ satisfies one of the following conditions:
\begin{enumerate}
\item
The canonical line bundle $\mathcal{K}_{Y_\mathrm{min}}$ is nef.
\item
$Y_\mathrm{min}$ is an $G$-equivariant conic bundle over a smooth curve,
i.e., there exists an $G$-equivariant morphism
$f: Y_{\mathrm{min}} \to C$ onto a smooth curve $C$ such that the general
fiber is a rational curve. If $f$ has singular fibers, these
consist of two (-1)-curves intersecting transversally.
\item
$\mathcal{K}^{-1}_{Y_{\mathrm{min}}}$ is ample, i.e., $Y_\mathrm{min}$ is a Del Pezzo surface.
\end{enumerate}
Each extremal contraction is the contraction of finitely many disjoint (-1)-curves forming a $G$-orbit.
\end{proposition}
The surface $Y_\mathrm{min}$ is referred to as a \emph{$G$-minimal model of $Y$},
and the map $Y \to Y_\mathrm{min}$ is called a \emph{Mori reduction}. 
 If $Y$ is a rational surface, then $Y_\mathrm{min}$ is either a Del Pezzo surface or an equivariant conic bundle over $\mathbb P_1$, i.e., the reduction leads to the well-known classification of $G$-minimal rational surfaces (\cite{maninminimal}, \cite{isk}).
\begin{remark}
Equivariant Mori reduction and the theory of $G$-minimal models have applications in various different contexts and can also be generalized to higher dimensions.
Initiated by Bayle and Beauville in \cite{Bayle}, the methods have been employed in the classification of subgroups of the Cremona group $\mathrm{Bir}(\mathbb P_2)$ for example by Beauville and Blanc (\cite{Beauville}, \cite{BeauBlancPrime}, \cite{PhDBlanc} \cite{Blanc1}), de Fernex \cite{fernex}, Dolgachev and Iskovskikh \cite{DolgIsk}, \cite{DolgIsk2}, and Zhang \cite{ZhangRational}. These references also provide details regarding the equivariant minimal model program summarized above.
\end{remark}
\subsection*{Del Pezzo surfaces}
In all cases under consideration in this article the surface $Y_\mathrm{min}$ is a Del Pezzo surface. 
Del Pezzo surfaces are classified according to their degree $1 \leq d \leq 9$, which is by definition the self-intersection number of the canonical divisor class. 
A Del Pezzo surface is either obtained from $\mathbb P_2$ by blowing up $9-d$ points in general position or is isomorphic to $\mathbb P_1 \times \mathbb P_1$.
The anticanonical map realizes a Del Pezzo surface of degree  $d \geq 3$ as a degree $d$ subvariety in $\mathbb P_d$, defines a double cover for $d=2$, and an elliptic fibration for $d=1$.

Our understanding of Del Pezzo surfaces as surfaces obtained by
blowing up points in $\mathbb P_ 2$ or as degree
$d$ subvarieties of $\mathbb P_d$ enables us to decide whether
a given finite group $G$ can occur as a subgroup of the automorphisms
group $\mathrm{Aut}(Z)$ of a certain Del Pezzo surface $Z$. Furthermore, we repeatedly consider the configuration of (-1)-curves on a Del Pezzo surfaces $Z$ and the induced action of $\mathrm{Aut}(Z)$ on it. 
In certain cases our analysis relies on Dolgachev's discussion of automorphism groups of Del Pezzo surfaces in \cite{dolgachev}, Chapter 10.

Using detailed knowledge of the equivariant reduction map $Y \to Y_\mathrm{min}$, the structure of the invariant set $\mathrm{Fix}_X(\sigma)$, 
and the equivariant geometry of Del Pezzo surfaces, we classify $Y_\mathrm{min}$, $Y$, and $\mathrm{Fix}_X(\sigma)$ and can describe $X$ as an equivariant double cover of a possibly blown-up Del Pezzo surface. 
\subsection*{Acknowledgements} 
The results presented in this note originate from the author's Ph.D.\ thesis at Ruhr-Universit\"at Bochum. The author would like to thank Alan Huckleberry for his advice and support and for numerous intensive, helpful discussions.

%
%
%
\section{Centralizers of antisymplectic involutions}\label{centralizer}
This section is dedicated to a rough classification (Theorem \ref{roughclassi}) of K3-surfaces with anti\-sym\-plec\-tic involutions centralized by large groups of symplectic transformations. For a group $H < \mathrm{Aut}(X)$ we denote by $H_\mathrm{symp} \lhd H$ the normal subgroup of symplectic automorphisms. The group $H / H_\mathrm{symp}$ acts effectively on $\Omega^2(X) = \mathbb C \cdot \omega_X$. In particular, if $H$ is finite, then $ H / H_\mathrm{symp}$ is cyclic.

We consider a K3-surface $X$ with an action of a finite group $G
\times C_2 < \mathrm{Aut}(X)$ and assume that $G < \mathrm{Aut}_\mathrm{symp}(X)$ whereas $C_2$ is generated by an
antisymplectic involution $\sigma$ centralizing $G$. Furthermore, we assume that $\mathrm{Fix}_X(\sigma) \neq \emptyset$. Let $\pi: X \to X/ \sigma = Y$ denote the quotient map, $R = \mathrm{Fix}_X(\sigma)$ and $B = \pi(R)$ the ramification and branch set of $\pi$. The quotient surface $Y$ is a smooth rational $G$-surface to which we apply an equivariant minimal model program. By Proposition \ref{Gmmp} a $G$-minimal model $Y_\mathrm{min}$ of $Y$ is either a Del Pezzo surface or an equivariant conic bundle over $\mathbb P_1$. In the later case, the following lemma shows that the possibilities for $G$ are limited by the classification of finite groups with an effective action on $\mathbb P_1$, i.e., cyclic and dihedral groups, $C_n$ and $D_{2n}$, and the exceptional tetrahedral, octahedral, and icosahedral group, $T_{12}$, $O_{24}$, $I_{60}$.
\begin{remark}\label{order of symp aut}
Several times we will use the fact that the order of a symplectic transformation on a K3-surfaces is at most eight. This follows implicitly from Table \ref{TableMukai}, but is originally due to an argument of Nikulin classifying finite
Abelian groups of symplectic transformation on K3-surfaces in \cite[\S5]{NikulinFinite}. Note that Nikulin' approach also shows that a symplectic automorphism of order 2,3,4,5,6,7,8 has 8,6,4,4,2,3,2 fixed points, respectively.
\end{remark}
\begin{lemma}\label{conicbundle}
If a $G$-minimal model $Y_\mathrm{min}$ of $Y$ is an equivariant conic bundle, then $|G| \leq 96$.
\end{lemma}
\begin{proof}
Let $\varphi: Y_\mathrm{min} \to \mathbb{P}_1$ be an equivariant conic bundle structure on $Y_\mathrm{min}$. We consider the induced action of $G$ on the base $\mathbb P_1$. If this action is effective, then $G$ is among the groups specified above. Since the maximal order of an element in $G$ is eight, 
it follows that the order of $G$ is bounded by 60.

If the action of $G$ on the base $\mathbb P_1$ is not effective, every element $n$ of the ineffectivity $N < G$ has two fixed points in the general fiber, which is isomorphic to $ \mathbb P_1$. This gives rise to a positive-dimensional $n$-fixed point set in $Y_\mathrm{min}$ and $Y$. A symplectic automorphism however has only isolated fixed points. It follows that the action of $n$ coincides with the action of $\sigma $ on $\pi^{-1}(\mathrm{Fix}_Y(N)) \subset X$ and the order of $n$ is two. Since $N$ acts effectively on the general fiber, it follows that $N$ is isomorphic to either $C_2$ or $C_2 \times C_2$. 

If $G/N$ is isomorphic to the icosahedral group $I_{60}= A_5$, then $G$ fits into the exact sequence $ 1 \to N \to G \to A_5 \to 1$ for $N=C_2$ or $C_2 \times C_2$. Let $\xi$ be an element of order five in $G$. Since neither $C_2$ nor $C_2 \times C_2$ has automorphisms of order five it follows that $\xi$ centralizes the normal subgroup $N$. In particular, there is a subgroup $C_2 \times C_5 \cong C_{10}$ in $G$ which contradicts the assumption that $G$ is a group of symplectic transformations and therefore its elements have order at most eight.

If $G/N$ is cyclic or dihedral, we again use the fact that the order of elements in $G$ is bounded by $8$ to conclude $|G/N| \leq 16$. It follows that the maximal possible order of $G/N$ is $|O_{24}|= 24$. Using $|N| \leq 4$ we obtain $|G| \leq 96$.
\end{proof}
For $| G | >96$ the lemma above allows us to restrict our classification to the case where any $G$-minimal model $Y_\mathrm{min}$ of $Y$ is a Del Pezzo surface.

\subsection{Branch curves and Mori fibers} \label{branch curves mori fibers}
We use the term \emph{curve} for irreducible 1-dimensional analytic subsets of surfaces.
For the remainder of this section we
fix an equivariant Mori reduction $M: Y \to Y_\mathrm{min}$.
A rational curve $E \subset Y$ is called a \emph{Mori fiber} if it is contracted in some step of the equivariant Mori reduction $Y \to Y_\mathrm{min}$. The set of all Mori fibers is denoted by $\mathcal E$. Its cardinality $|\mathcal E|$ is denoted by $m$.

A linearization argument shows that $\mathrm{Fix}_X(\sigma)$ is a disjoint union of smooth curves. More precisely, by a basic result originally due to Nikulin (\cite{NikulinFix}), the set $\mathrm{Fix}_X(\sigma)$ is either a union of two linearly equivalent elliptic curves or the union of one single curve of arbitrary genus with a (possibly empty) union of rational curves. In the following we let $n$ denote the total number of rational curves in $\mathrm{Fix}_X(\sigma)$. 
We choose a triangulation of $\mathrm{Fix}_X(\sigma)$ and extend it to a triangulation of the surface $X$. The topological Euler characteristic of the double cover is given by
\[
 e(X) = 24 = 2e(Y) - \sum_{C \subset \mathrm{Fix}_X(\sigma)} e(C)\geq 2e(Y) -2n  = 2(e(Y_\mathrm{min}) +m) -2n.
\]
This yields $m \leq n+12 - e(Y_\mathrm{min})$. Combining this with $e(Y_{min}) \geq 3$ shows that the total number $m$ of Mori fibers in $Y$ is bounded by 
\begin{equation}\label{moribound}
 m \leq n+ 12 - e(Y_\mathrm{min}) \leq n+9. 
\end{equation}
\begin{remark}\label{selfintcovering}
In the following, we repeatedly use the fact that for a finite proper surjective holomorphic map of complex manifolds (spaces) $\pi: X \to Y$ of degree $d$, the intersection number of pullback divisors fulfills $(\pi^*D_1 \cdot \pi^*D_2) =  d(D_1 \cdot D_2)$. In particular, if $\pi: X \to Y$ is a degree two map of surfaces, the image $\pi (C) \subset Y$ of a curve $C \subset X$ contained in the ramification locus $R$ of $\pi$ has self-intersection $(\pi(C))^2= 2 C^2$.
\end{remark}
\begin {lemma} \label {at most two}
Every Mori fiber $E \in \mathcal{E}$, $ E \not\subset B$ meets the branch locus $B$ in at
most two points. If $E$ and $B$ are tangent at $p$, then 
$E\cap B = \{p\}$ and $(E \cdot B)_p =2$.
\end {lemma}
\begin{proof} 
Let $E \in \mathcal{E}$, $E^2=k$ be a Mori fiber such that $E \not\subset B$ and $|E\cap B| \geq 2$ or $E \cdot B \geq 3$.
By the remark above, the divisor
$\pi^{-1}(E) = \pi^* E$ has self-inter\-section $2k$. Assume that $\pi^{-1}(E)$ is
reducible and let $\tilde E_1, \tilde E_2$ denote its irreducible components. These are rational
and therefore, by adjunction on the K3-surface $X$, have self-intersection number $-2$.  We write
$$
0 > 2k = (\pi^{-1}(E))^2 = \tilde E_1^2 + \tilde E_2^2 + 2 (\tilde E_1 \cdot \tilde E_2) = -4 + 2 (\tilde E_1 \cdot \tilde E_2).
$$
Since $\tilde E_1$ and $\tilde E_2$ intersect at points in the preimage of $E \cap B$, we obtain $\tilde E_1 \cdot \tilde E_2 \geq 2$, a contradiction. It follows that $\pi^{-1}(E)$ is irreducible. Consequently, $k=-1$ and $\pi^{-1}(E)$ is a smooth rational curve with precisely two $\sigma$-fixed points showing $|E\cap B| = 2$. It remains to show that the intersection is transversal.

To see this, let $N_{\tilde{E}}$ denote the normal bundle of $\tilde{E}$ in $X$. We
consider the induced action of $\sigma$ on $N_{\tilde{E}}$ by a bundle
automorphisms. We may equivariantly identify a tubular
neighbourhood of $\tilde E$ in $X$ with $N_{\tilde E}$ via a 
$C^{\infty}$-diffeomorphism. The $\sigma$-fixed point curves
intersecting $\tilde{E}$ map to curves of $\sigma$-fixed points in
$N_{\tilde{E}}$ intersecting the zero-section and vice versa. 
Let $D$ be a curve of $\sigma$-fixed points in $N_{\tilde{E}}$. If $D$ is
not a fiber of $N_{\tilde E}$, it follows that $\sigma$ stabilizes all 
fibers intersecting
$D$ and the induced action of $\sigma$ on the base must be trivial, a contradiction.
It follows that the $\sigma$-fixed point curves correspond to fibers of 
$N_{\tilde{E}}$, and $E$ and $B$ meet transversally.

By negation of the implication above, if $E$ and $B$ are tangent at $p$, then $|E \cap B |=1$ and $E \cdot B=2$.
\end{proof}
\begin{remark}\label{self-int of Mori-fibers}
Adjunction on $X$ and elementary considerations involving intersection numbers resembling those in the proof of the lemma above yield the following summary of the possible intersection geometry of a Mori fiber $E\in \mathcal E$ with 
the branch locus $B$ in relation to the self-intersection number of $E$.
\\
If $E \subset B $, then $E^2 = -4$. If $E \cap B = \emptyset$, then $E^2 = -2$ and $\pi^{-1}(E)$ is reducible.
If $E\cap B = \{p\}$, then $E^2 = -1$ and $\pi^{-1}(E)$ is reducible. If $E\cap B = \{p_1,p_2\}$, then $E^2 = -1$ and $\pi^{-1}(E)$ is irreducible.
More generally, the arguments involved apply to any (-1)-curve $E$ on $Y$ and therefore $E$ meets $B$ in either one or two points and $\pi^{-1}(E)$ is reducible or irreducible, respectively.
\end{remark}
\subsection{Rational branch curves}
In this subsection we find conditions on $G$, in particular conditions on the order of $G$, guaranteeing the absence of rational curves in $\mathrm{Fix}_X(\sigma)$. 

 It follows from adjunction that a curve with negative self-intersection on a Del Pezzo surface necessarily has self-intersection -1. So if $Y_\mathrm{min}$ is a Del Pezzo surface, all rational branch curves of $\pi$, which have self-intersection -4 by Remark \ref{selfintcovering}, need to be modified by the Mori reduction when passing to $Y_{\mathrm{min}}$ and therefore have nonempty intersection with the union of Mori fibers. We benefit from the following elementary observation regarding the behaviour of self-intersection numbers under monoidal transformations.
\begin{remark}\label{selfintblowdown} 
Let $\tilde X$ and $X$ be smooth projective surfaces and let $b: \tilde X \to X$ be the blow-down of a (-1)-curve $ E \subset \tilde X$. For a curve $B \subset \tilde X$, $B \neq E$, the self-intersection of its image in $X$ is given by $(b( B))^2 =  B^2 + ( E \cdot B)^2$.
\end{remark}
We denote by $\mathcal{C}$ the set of rational branch curves of
$\pi$. The total number $|\mathcal{C}|$ of these curves is denoted
by $n$. The union of all Mori fibers not contained in the branch locus $B$ is denoted by $\bigcup E_i$.
Let $\mathcal{C}_{\geq k}= \{ C \in \mathcal C \, | \, |C \cap  \bigcup E_i | \geq k\} $ be the set of those rational branch curves $C$ which meet $\bigcup E_i$ in at least $k$ distinct points and set $r_k = |\mathcal{C}_{\geq k}|$. 
We denote by
$\mathcal{E}_{\geq k}$ the set of Mori fibers $E \not\subset B$ which
intersect some $C \in \mathcal{C}_{\geq k}$ and define 
\[
P_k = \{ (p,E) \, |\, p \in C \cap E,\, E \in \mathcal{E}_{\geq k},\, C
\in \mathcal{C}_{\geq k}\} \subseteq Y \times \mathcal{E}_{\geq k}
\]
and the projection map $\mathrm{pr}_k: P_k \to \mathcal{E}_{\geq k}$ mapping
$(p,E)$ to $E$. This map is surjective by definition of
$\mathcal{E}_{\geq k}$ and its fibers consist of $\leq 2$ points by
Lemma \ref{at most two}. Using $|P_k| \geq k r_k$ we see
\begin{equation}\label{boundforE_k}
|\mathcal{E}_{\geq k}| \geq \frac{k}{2}r_k.
\end{equation} 
Let $N$ be the largest positive integer such that $\mathcal{C}_{\geq N} =
\mathcal{C}$, i.e., each rational ramification curve is intersected at least $N$
times by Mori fibers. 
A curve $C \in \mathcal{C}$ which is intersected precisely $N$ times by
Mori fibers is referred to as a \emph{minimizing curve}. In the
following, let $C$ be a minimizing curve and let $ H =
\mathrm{Stab}_G(C) < G$ be the stabilizer of $C$ in $G$.
The index of $H$ in $G$ is bounded by $n= r_N$.

The rational curves
in  $\mathrm{Fix}_X(\sigma)$ generate a sublattice of
$\mathrm{Pic}(X)$ of signature $(0,n)$ and it therefore follows immediately that $n
\leq 19$.
A sharper bound $n \leq 16$ for the number of disjoint (-2)-curves on a K3-surface has been
obtained by Nikulin \cite{NikulinKummer}.  In our setup an even sharper bound is due to Zhang \cite[Theorem 3]{ZhangInvolutions}, stating that
the total number of connected curves in the fixed point set of an antisymplectic
involution on a K3-surface is bounded by 10.
In the following, we use Zhang's bound for the number $n$ of rational curves in $\mathrm{Fix}_X(\sigma)$,
\begin{equation}\label{atmostten}
n \leq 10.
\end{equation}
Note, however, that all results can likewise be obtained by using the weakest bound $n \leq 19$. 
For $N \geq 4 $ Zhang's bound can be sharpened using the notion of
Mori fibers and minimizing curves.
Using inequalities \eqref{moribound} and \eqref{boundforE_k} we find
\begin{equation}\label{boundforn}
\frac{N}{2}n = \frac{N}{2}r_N \leq |\mathcal{E}_{\geq N}|\leq
|\mathcal{E}| \leq  n +12 - e(Y_\mathrm{min})  \leq n+9
\end{equation}
In the following we consider the stabilizer $H$ of a minimizing curve
$C$ and using the above inequalities for $n$, we obtain bounds for the order $|G|$
of $G$ guaranteeing the absence of rational curves in $\mathrm{Fix}_X(\sigma)$.
\begin{proposition}
Let $X$ be a K3-surface with an action of a finite group $G \times \langle \sigma
\rangle$ such that $ G < \mathrm{Aut}_\mathrm{symp}(X)$ and $\sigma$ is
an antisymplectic involution with fixed points. If $| G | > 108$,
then $\mathrm{Fix}_X(\sigma)$ does not contain rational curves.
\end{proposition}
\begin{proof}
 Let  $| G | > 108$ and assume on the contrary that $\mathrm{Fix}_X(\sigma)$ contains rational curves.
 We consider a minimizing curve $C \subset B$ and its stabilizer $\mathrm{Stab}_G(C) =:H$. Since a symplectic automorphism on $X$ does not admit a one-dimensional set of fixed points, it follows that the action of $H$ on $C$ is effective and $H$ is one the classical finite groups $C_n, D_{2n}, T_{12}, O_{24}, I_{60}$.
We recall the possible lengths of $H$-orbits in $C$: the length of an orbit of a dihedral group is at least two, the length of a $T_{12}$-orbit in $\mathbb P_1$ is at least four, the length of an $O_{24}$-orbit in $\mathbb P_1$ is at least six, and the length of an $I_{60}$-orbit in $\mathbb P_1$ is at least twelve.

 Let $Y_\mathrm{min}$ be a $G$-minimal model of $X/\sigma =Y$. Recall that by Lemma \ref{conicbundle} the surface $Y_\mathrm{min}$ is a Del Pezzo surface. Each rational branch curve is a (-4)-curve in $Y$. Since its image in $Y_\mathrm{min}$ has self-intersection $\geq -1$, it must intersect Mori fibers.
\begin{itemize}
 \item 
If $N=1$, i.e., the rational curve $C$ meets the union of Mori fibers in exactly one point $p$, then $p$ is a fixed point of the $H$-action on $C$. In particular, $H$ is a cyclic group $C_k$ and $k \leq 8 $ by Remark \ref{order of symp aut}. Since the index of $H$ in $G$ is bounded by $n \leq 10$, it follows that $|G | \leq 80$.
\item
If $N=2$, then $H$ is either a cyclic or a dihedral group. By Proposition 3.10 in \cite{mukai} the maximal order of a dihedral group of symplectic automorphisms on a K3-surface is 12. We first assume $H \cong D_{2m}$ and that the $G$-orbit $G.C$ of the rational branch curve $C$ has the maximal length $n=|G.C|=10$, i.e., $B = G . C$. 
Each curve in $G . C$ meets the union of Mori fibers in precisely two points forming an $D_{2m}$-orbit. If a Mori fiber $E_C$ meets the curve $C$ twice, then it follows from Lemma \ref{at most two} that $E$ meets no other curve in $B$. The contraction of $E$ transforms $C$ into a singular curve of self-intersection zero. The Del Pezzo surface $Y_\mathrm{min}$ does however not admit a curve of this type. It follows, that $E$ meets a Mori fiber $E'$ which is contracted in a later step of the Mori reduction and meets no other Mori fiber than $E'$.
The described configuration $G. E \cup G. E'$ requires a total number of at least 20 Mori fibers and therefore contradicts inequality \eqref{moribound}. If $C$ meets two distinct Mori fibers $E_1, E_2$, each of these two can meet at most one further curve in $B$. The contraction of $E_1$ and $E_2$ transforms $C$ into a (-2)-curve. As above, the existence of further Mori fibers meeting $E_i$ follows. Again, by invariance, the total number of Mori fibers exceeds 20, a contradiction. It follows that either $H$ is cyclic or $ |G.C|\leq 9$. Both cases imply $|G| \leq 108$. 
\item
If $N=3$, let $S=\{p_1,p_2,p_3\}$ be the points
of intersection of $C$ with the union $\bigcup E_i$ of Mori fibers. The set $S$ is
$H$-invariant. It follows that $H$ is either trivial or isomorphic to $C_2$, $C_3$
or $D_6$ and that $|G| \leq 60$
\item
If $N = 4$, it follows from inequality \eqref{boundforn} that $n \leq 9 $. Now $|H| \leq 12$ implies $|G| \leq 108$. 
\item
If $N=5$, the largest possible group acting on $C$ such
that there is an invariant subset of car\-di\-na\-li\-ty 5 is the dihedral
group $D_{10}$. Inequality \eqref{boundforn} implies $n \leq 6$, we conclude $|G| \leq 60$.
\item 
If $N=6$, then  $n \leq 4 $ and $|H| \leq 24$ implies $|G| \leq
96$. 
\item
If $N\geq 12$, then $n=1$ and $H =G$. The maximal order 60 is attained by
the icosahedral group. 
\item
If $6 < N <12$,
we combine $n \leq 4$ and $|H| \leq 24$ to obtain
$|G| \leq 96$.
\end{itemize}
The case by case discussion shows that the existence of a rational curve in $B$ implies $| G|\leq 108$ and  contradicts our assumption. The proposition follows.
\end{proof}
\begin{remark}
If the group $G$ under consideration is known not to contain certain subgroups (such as large dihedral groups or $T_{12}$, $O_{24}$ or $I_{60}$), then the condition $|G| >108$ in the proposition above can be improved and non-existence of rational ramification curves follows even for smaller groups $G$ (cf.\ Lemmata \ref{noratN72}, \ref{subgroupsM9}, and \ref{no rat T48}).
\end{remark}
\subsection{Elliptic branch curves}
The aim of this section is to determine conditions on the order of $G$ which allow us to exclude elliptic curves in $\mathrm{Fix}_X(\sigma)$. We prove:
\begin{proposition}\label{elliptic branch}
Let $X$ be a K3-surface with an action of a finite group $G \times \langle \sigma
\rangle$ such that $ G < \mathrm{Aut}_\mathrm{symp}(X)$ and $\sigma$ is
an antisymplectic involution with fixed points. If $| G | > 108$, 
then $\mathrm{Fix}_X(\sigma)$ contains neither rational nor elliptic ramification curves.
\end{proposition}
\begin{proof}
By the preceding proposition $\mathrm{Fix}_X(\sigma)$ does not contain rational curves. It follows from Ni\-ku\-lin's description of $\mathrm{Fix}_X(\sigma)$ in \cite{NikulinFix} that it is either a single curve of genus $g \geq 1$ or the disjoint union of two elliptic curves.

Let $T \subset B$ be an elliptic branch curve and let $H:= \mathrm{Stab}_G(T)$. If $H \neq G$, then $H$ has index two in $G$. The action of $H$ on $T$ is effective.
After conjugation the subgroup $H < \mathrm{Aut}(T)$ inherits the semidirect
product structure of $\mathrm{Aut}(T) = L \ltimes T$ for $L \in \{ C_2, C_4, C_6\}$. I.e.,
\[
H = (H \cap L) \ltimes  (H \cap T) .
\]
We refer to this decomposition as the \emph{normal form} of $H$.
By Lemma \ref{conicbundle} any $G$-minimal model of $Y$ is a Del Pezzo surface and therefore, by adjunction, does not admit elliptic curves with self-intersection zero. It follows that $T$ meets the union $\bigcup E_i$ of Mori fibers. Let $E$ be a Mori fiber meeting $T$. By Lemma \ref{at most two} their intersections fulfills $|T \cap  E |  \in \{1,2\}$. Since the total number of Mori fibers is bounded by 9 (cf.\ inequality \eqref{moribound}), the index of the stabilizer $\mathrm{Stab}_H(E)$ of $E$ in $H$ is bounded by 9.
If $T \cap E =\{p\}$, then $\mathrm{Stab}_H(E)$ is a cyclic group of order less than or equal to six and it follows that $|G | \leq 6\cdot 9 \cdot 2 = 108$.
If $T \cap E =\{p_1, p_2\}$, then $B \cap E = T \cap E$ and the stabilizer $\mathrm{Stab}_G(E)$ of $E$ in $G$ is contained in $H$. If both points $p_1,p_2$ are fixed by $\mathrm{Stab}_G(E)$, then $|\mathrm{Stab}_G(E)| \leq 6$. If $p_1,p_2$ form a $\mathrm{Stab}_G(E)$-orbit, then in the normal form $|\mathrm{Stab}_G(E) \cap T | =2$. It follows that $\mathrm{Stab}_G(E)$ is either $C_2$ or $D_4$. The index of $\mathrm{Stab}_G(E)$ in $G$ is bounded by 9 and $|G| \leq 54$.

In summary, the existence of an elliptic curve in $B$ implies $|G| \leq 108$.
\end{proof}
\subsection{Rough classification}
With the preparations of the previous subsections we may now turn to a classification result for K3-surfaces with antisymplectic involution centralized by a large group.
\begin{theorem}\label{roughclassi}
Let $X$ be a K3-surface with a symplectic action of a finite group $G$ centralized by an
antisymplectic involution $\sigma$ such that
$\mathrm{Fix}_X(\sigma)\neq \emptyset$. If $|G|>96$, then $X/\sigma$ is a $G$-minimal
Del Pezzo surface and $\mathrm{Fix}_X(\sigma)$ is a smooth connected curve $B$ with $ g(B)\geq
3$.
\end{theorem}
\begin{proof}
The group $G$ is a subgroup of one of the eleven maximal finite groups of symplectic transformations on Mukai's list \cite{mukai} (cf.\ Table \ref{TableMukai}).
None of these groups can have a subgroup $G$ with $96 < |G| < 120$.
In particular, the order of $G$ is at least 120.
We may therefore apply the results of the previous two sections and conclude that $\pi: X \to X/\sigma = Y$ is branched along a single smooth curve $B$ of general type. Its genus $g(B)$ must be greater than or equal to three by Hurwitz's formula.
It remains to show that $Y$ is $G$-minimal.

Assume on the contrary that $Y$ is not $G$-minimal. Then there exists a
Mori fiber $E \subset Y$ with $E^ 2= -1$.
By Remark \ref{self-int of Mori-fibers} it intersects the branch curve $B$ in one or two points. Let $\mathrm{Stab}_G(E)$ denote the stabilizer of $E$ in $G$. 

If $\pi^{-1}(E)$ is reducible its two irreducible components meet transversally in one point corresponding to $\{p\} =E \cap B$. The curve $E$ is tangent to $B$ at $p$ and we consider the linearization of the action of $\mathrm{Stab}_G(E)$ at $p$. If the action of $\mathrm{Stab}_G(E)$ on $E$ is not effective, the linearization of the ineffectivity $I < \mathrm{Stab}_G(E)$ yields a trivial action of $I$ on the tangent line of $B$ at $p$. It follows that the action of $I$ is trivial in a neighbourhood of $\pi^{-1}(p) \in R =\pi^{-1}(B)$. This is contrary to the assumption that $G$ acts symplectically on $X$.  
Consequently, the action of  $\mathrm{Stab}_G(E)$ on $E$ is effective and in particular, $\mathrm{Stab}_G(E)$ is a cyclic group. 

If $\pi^{-1}(E)$ is irreducible, then it is a smooth rational curve with an effective action of $\mathrm{Stab}_G(E)$. It follows that $\mathrm{Stab}_G(E)$ is either cyclic or dihedral. 

We conclude that the order of  $\mathrm{Stab}_G(E)$ is bounded by 12 and the index of $G_E$ in $G$ is strictly greater than nine. By inequality \eqref{moribound} the total number $m$ of Mori fibers however satisfies $m \leq 9$. This contradiction shows that $Y$ is $G$-minimal and, in particular, a Del Pezzo surface. 
\end{proof}
\begin{remark}\label{stab of minus one curve}
Let $X$ be a K3-surface with a symplectic action of $G$ centralized by an
antisymplectic involution $\sigma $ with $\mathrm{Fix}_X(\sigma) \neq \emptyset$ and let $E$ be a (-1)-curve on $Y = X/\sigma$. Then the argument above can be applied to see that the stabilizer of $E$ in $G$ is cyclic or dihedral and therefore has order at most 12.
\end{remark}
In the following section, the classification above is applied and extended to the case where $G$ is a maximal group of symplectic transformations on a K3-surface. 
%
%
%
\section{Maximal finite groups of symplectic transformations}\label{maximal}
In this section we consider K3-surfaces with a symplectic action of one of the eleven groups from Mukai's list (Table \ref{TableMukai}) centralized by an antisymplectic involution and prove the following classification result.
\begin{theorem}\label{thm mukai times invol}
Let $G$ be a group from Table \ref{TableMukai} acting on a K3-surface $X$ by symplectic transformations and $\sigma $ be an antisymplectic involution on $X$ centralizing $G$ with $\mathrm{Fix}_X(\sigma) \neq \emptyset$. Then the pair $(X,G)$ is equivariantly isomorphic to a surface in Table \ref{Mukai times invol}. In particular, for the groups $G$ numbered 4-8 on Mukai's list, there does not exist a K3-surface with an action of $G \times C_2$ satisfying the properties above. 
\end{theorem}
For the proof of this theorem we consider each group separately, show that any $G$-minimal model $Y_\mathrm{min}$ of the quotient surface $Y = X / \sigma$ is a Del Pezzo surface, and investigate which Del Pezzo surfaces admit an action of the group $G$.

It is then essential to study the branch locus $B$ of the covering $X \to Y$. As a first step, we exclude rational and elliptic curves in $B$ by studying their images in $Y_\mathrm{min}$ and their intersection with the union of Mori fibers. 
We then deduce that $B$ consists of a single curve of genus $\geq 2$ with an effective action of the group $G$. The possible genera of $B$ are restricted by the nature of the group $G$ and the Riemann-Hurwitz formula for the quotient of $B$ by an appropriate normal subgroup $N$ of $G$. The equations of $B$ or $X$ given in Table \ref{Mukai times invol} are derived using invariant theory.

Throughout the remainder of this chapter, the Euler characteristic formula
\begin{equation}\label{eulerchar}
 24 = e(X) = 2 e(Y_\mathrm{min}) + 2 m -2n +\underset{\text{if branch curve with }g\geq 2 \text{ exists}}{\underbrace{(2g-2)}}
\end{equation}
is exploited at various points.
Here $m$ denotes the total number of Mori contractions of the reduction $Y \to Y_ \mathrm{min}$, the total number of rational branch curves is denoted by $n$ and $g$ is the genus of a non-rational branch curve.
\subsection*{Equivariant equivalence}
Let us briefly formalize the notion of equivariant equivalence.
\begin{definition}\label{equivariantequivalence}
Let $(X_1, \sigma_1)$ and $(X_2, \sigma_2)$ be K3-surfaces with antisymplectic involution and let $G$ be a finite group acting on $X_1$ and $X_2$ by 
$\alpha_i: G \to \mathrm{Aut}_\mathrm{symp}(X_i)$,
such that $\alpha_i(g) \circ \sigma_i = \sigma_i \circ \alpha_i(g)$ for $i =1,2$ and all $g \in G$.
Then the surfaces $(X_1, \sigma_1)$ and $(X_2, \sigma_2)$ are considered \emph{equivariantly equivalent} or \emph{equivariantly isomorphic} if there exist a biholomorphic map $\varphi: X_1 \to X_2$ and a group automorphism $\psi \in \mathrm{Aut}(G)$ such that
\[
 \alpha_2(g) \varphi (x) = \varphi ( \alpha_1(\psi(g))x) \quad \text{and} \quad \sigma_2(\varphi (x)) = \varphi ( \sigma_1(x)).
\]
for all $x \in X_1$ and all $g \in G$.

Two (rational) surfaces $Y_1$ and $Y_2$
with actions 
$\alpha_i: G \to \mathrm{Aut}(Y_i)$ of a finite group $G$
are considered equivariantly equivalent if there exist a biholomorphic map $\varphi: Y_1 \to Y_2$ and a group automorphism $\psi \in \mathrm{Aut}(G)$ such that
$\alpha_2(g) \varphi (y) = \varphi ( \alpha_1(\psi(g))y)$
for all $y \in Y_1$ and all $g \in G$.
\end{definition}
\begin{remark}
If two K3-surfaces $(X_1, \sigma_1)$ and $(X_2, \sigma_2)$ are $G$-equivariantly equivalent, then the quotient surfaces $X_i / \sigma_i$ are equivariantly equivalent with respect to the induced action of $G$. 
\\
Conversely, let $Y$ be a rational surface with two actions of a finite group $G$ which are equivalent in the above sense and let $\varphi \in \mathrm{Aut}(Y)$ be the isomorphism identifying these two actions. We consider a smooth $G$-invariant curve $B$ linearly equivalent to $-2K_Y$ and the K3-surfaces $X_B$ and $X_{\varphi(B)}$ obtained as double covers branched along $B$ and $\varphi(B)$ equipped with their respective antisymplectic covering involution. 
If all elements of the group $G$ can be lifted to symplectic transformations on $X_B$ and $X_{\varphi(B)}$, then the central degree two extensions $E$ of $G$ acting on $X_B$, $X_{\varphi(B)}$, respectively, split as $E = E_\mathrm{symp} \times C_2$ with $E_\mathrm{symp} =G$.
In this case the group $G$ acts by symplectic transformations on 
$X_B$ and $X_{\varphi(B)}$ and
since for each $g \in G \subset \mathrm{Aut}(Y)$ 
there is only one choice of symplectic lifting $\tilde g \in E$
 the surfaces $X_B$ and $X_{\varphi(B)}$ are $G$-equivariantly equivalent in strong sense introduced above.
\end{remark}
%
\subsection{The group \texorpdfstring{$L_2(7)$}{L2(7)}}
Let $G \cong L_2(7) =\mathrm{PSL}(2, \mathbb F_7)=\mathrm{GL}_3(\mathbb F_2)$ be the finite simple group of order 168 acting on a K3-surface $X$. As $G$ is simple this action is effective and symplectic. Let $\sigma$ be an antisymplectic involution on $X$ centralizing $G$. Since $G$ has an element of order seven, which by Remark \ref{order of symp aut} has precisely three fixed points $p_1, p_2,p_3$ in $X$, and $\sigma$ acts on this set of three points, we know that $\mathrm{Fix}_X(\sigma) \neq \emptyset$.
I.e., the assumption $\mathrm{Fix}_X(\sigma) \neq \emptyset$ in Theorems \ref{roughclassi} and \ref{thm mukai times invol} is implicitly true for this particular group.

The classification result for the group $G = L_2(7)$ can be deduced from a description of K3-surfaces with $C_3 \ltimes C_7$-symmetry obtained in \cite{FKH1}, Theorem 2 (cf.\ also Chapter 5 in \cite{kf}, Theorem 5.5). Using the same arguments but restricting our consideration to the larger group $L_2(7)$ we present a simplified proof of the following proposition.
\begin{proposition}\label{classiL2(7)}
 Let $X$ be a K3-surface with a symplectic action of the group $L_2(7)$ centralized by an antisymplectic involution $\sigma$. Then $X$ is equivariantly isomorphic to either $C_\mathrm{Klein} = \{x_1^3x_2+x_2^3x_3+x_3^3x_1+x_4^4 =0\} \subset \mathbb P_3$ or the
double cover of $\mathbb P_2$ branched along $\mathrm{Hess}(C_\mathrm{Klein})$.
\end{proposition}
We begin with a classification of the quotient surface $Y = X / \sigma$.
\begin{lemma}
The quotient surface $Y$ is either $\mathbb P_2$ or a Del Pezzo surface of degree 2.
\end{lemma}
\begin{proof}
We may apply Theorem \ref{roughclassi} to conclude that $Y$ is a $G$-minimal Del Pezzo surface and denote 
by $d = \mathrm{deg}(Y) = K_Y^2$ its degree.

If $d=1$, then the anticanonical system $|-K_Y|$ is known to have precisely one base point. In particular, this point is fixed with respect to the full automorphism group.
Since $G$ does not admit a two-dimensional representation, this yields a contradiction.
As there is no injective homomorphism $G \hookrightarrow \mathrm{Aut}(\mathbb P_1 \times \mathbb
P_1) = (\mathrm{PSL}_2(\mathbb C) \times \mathrm{PSL}_2(\mathbb C)) \rtimes C_2$ and since the blow-up of $\mathbb P_2$ in one or two points admits an equivariant contraction map to either $\mathbb P_2$ or $\mathbb P_1 \times \mathbb P_1$ and is therefore never $G$-minimal, it follows that $d \neq 7,8$.  

For the remaining Del Pezzo surfaces we consider the action of $G$ on their configurations of (-1)-curves. 
Noting that the maximal
subgroups of $G$ are $C_3 \ltimes C_7$ and $S_4$ of index 8 and 7, respectively, we find that the $G$-orbit of a (-1)-curve $E
\subset Y$ consists of 7, 8, 14, 21, 24, 28 or more curves. Since the number of (-1)-curves on a Del Pezzo surface $Y$ of degree 3, 5, 6  equals  27, 10, 6, respectively, it follows that $d \not \in \{3,5,6\}$ .
If $d=4$, then the
union of (-1)-curves on $Y$ consists of two $G$-orbits of length 8. In
particular, $\mathrm{Stab}_G(E)\cong C_3 \ltimes C_7$ for any
(-1)-curve $E \subset Y$. 
Blowing down $E$ to a point $p \in Y'$ induces an action of $C_3 \ltimes C_7$ on $Y'$ fixing $p$. Since $C_3 \ltimes C_7$ does not admit a two-dimensional representation, it follows that the normal subgroup $C_7$ acts trivially on $Y'$ and therefore on $Y$. This is a contradiction and completes the proof of the lemma.
\end{proof}
It remains to classify the branch locus $B$ in both cases. 
%
\subsection*{Double covers of \texorpdfstring{$\mathbb P_2$}{P2}}
The action of $L_2(7)$ on $\mathbb P_2$ is necessarily given by a three-dimensional represention.
This follows from the fact that the group $L_2(7)$ does not admit nontrivial degree three central extensions. This can be derived from the cohomology group $H^2(L_2(7), \mathbb C^*) \cong C_2$ known as the Schur Multiplier. 
There are two isomorphism classes of three-dimensional representations and these differ by an outer automorphism. We may therefore consider one particular representation and check explicitly that in appropriately chosen coordinates the curve $B = \mathrm{Hess}(C_\mathrm{Klein}) = \{x_1^5x_2+x_3^5x_1+x_2^5x_3-5x_1^2 x_2^2 x_3^2 =0\}$ is $L_2(7)$-invariant. As the notation indicates this sextic is the Hessian curve associated to Klein's quartic curve $C_\mathrm{Klein}= \{ x_1^3x_2+x_2^3x_3+x_3^3x_1 =0\}$

In order to show that $B$ is the unique $G$-invariant sextic in $\mathbb P_2$ we consider the action of $G$ on $B$.
The maximal possible isotropy group is $C_7$ and each $G$-orbit in  $B$ consists of at least 21 elements. If there was another $G$-invariant smooth sextic curve $C \subset \mathbb P_2$, then the invariant set $ B \cap C$ would consist of at most 36 points. This is a contradiction.

Let $X$ be the K3-surface obtained as a double cover of $\mathbb P_2$ branched along $\mathrm{Hess}(C_\mathrm{Klein})$. 
It remains to check that $G$ lifts to a subgroup of $\mathrm{Aut}(X)$: On $X$ we find an action of a central degree two extension $E$ of $L_2(7)$. Since $L_2(7)$ is simple, it follows that $E_\mathrm{symp} \vartriangleleft E$ is mapped onto $L_2(7)$ and $E_\mathrm{symp} \cong L_2(7)$. In particular, the group $E$ splits as $E_\mathrm{symp} \times C_2$ where $C_2$ is generated by the antisymplectic covering involution. 
We have shown:
\begin{proposition}
 Let $X$ be a K3-surface with a symplectic action of the group $L_2(7)$ centralized by an antisymplectic involution $\sigma$. If $Y=X/\sigma \cong \mathbb P_2$ then $X$ is equivariantly isomorphic to the
double cover of $\mathbb P_2$ branched along $\mathrm{Hess}(C_\mathrm{Klein})$.
\end{proposition}
\subsection*{Double covers of Del Pezzo surfaces of degree two}
Here we assume that the degree of $Y$ equals two and consider the anticanonical map $Y \to \mathbb P_2$. It has degree two, is branched along a smooth curve of degree four, and equivariant with respect to $\mathrm{Aut}(Y)$. 
We obtain an action of $G$ on $\mathbb P_2$ stabilizing a smooth quartic. Choosing one particular 3-dimensional representation of $G$ one checks by direct computation that the smooth curve $C_\mathrm{Klein}= \{ x_1^3x_2+x_2^3x_3+x_3^3x_1 =0\}$ is invariant. It follows from classical invariant theory or from the intersection number argument applied above that $C_\mathrm{Klein}$ is the unique quartic curve in $\mathbb P_2$ with this property.
The curve $C_\mathrm{Klein}$ is a prominent example since its automorphism group $\mathrm{Aut}(C_\mathrm{Klein}) = G$ attains the maximal possible order $168 = 84(g(C_\mathrm{Klein}) -1)$ allowed by the Hurwitz formula.

Consider the branch curve $B \subset Y$ of the covering $\pi: X \to Y$ and the preimage $C \subset Y$ of $C_\mathrm{Klein}$. By construction, both curves belong to the linear system $|-2K_Y|$. If they do not coincide, they meet in $(-2K_Y)^2 = 8$ or less points. This contradicts the fact that the set $B \cap C$ must be $G$-invariant and therefore consists of at least 21 points in a $G$-orbit.

It follows that the surface $X$ is a cyclic degree four cover of $\mathbb P_2$ branched along $C_\mathrm{Klein}$ and can equivalently be described as the quartic hypersurface $\{x_1^3x_2+x_2^3x_3+x_3^3x_1+x_4^4 =0\} \subset \mathbb P_3$. On $X$ there is action of the group $\tilde G = G \times C_4$ where  $\tilde G_\mathrm{symp} = G$. We have shown:
\begin{proposition}
 Let $X$ be a K3-surface with a symplectic action of the group $L_2(7)$ centralized by an antisymplectic involution $\sigma$. If $Y=X/\sigma$ is a Del Pezzo surface of degree two then $X$ is equivariantly isomorphic to the
cyclic degree 4 cover of $\mathbb P_2$ branched along $C_\mathrm{Klein}$, i.e., Mukai's $L_2(7)$ example $\{x_1^3x_2+x_2^3x_3+x_3^3x_1+x_4^4 =0\} \subset \mathbb P_3$.
\end{proposition}
Combining this with the previous proposition yields the classification result summarized in Proposition \ref{classiL2(7)}. 
\begin{remark}
Proposition \ref{classiL2(7)} yields an alternative proof of the Main Theorem in \cite{OZ168} classifying K3-surfaces with actions of finite groups containing $G = L_2(7)$ as a subgroup of index four (cf.\ Corollary 1.3 in \cite{FKH1} and Theorem 6.1 and Lemma 6.2 in \cite{kf})
\end{remark}
\begin{remark}
In \cite{FKH1} K3-surfaces with a symplectic action of $C_3 \ltimes C_7$ centralized by an antisymplectic involution are described as double cover of $\mathbb P_2$ branched along sextic curves in a 1-dimensional family $\mathcal M$. Up to equivalence, there is a unique singular curve $C_\mathrm{sing}$ in this family. The K3-surface obtained as the minimal desingularization of the corresponding double cover of $\mathbb P_2$ is, by contruction, a double cover of the blow-up $b : Y \to \mathbb P_2$ at the seven singular points of $C_\mathrm{sing}$. The surface $Y$ is a Del Pezzo surface of degree two, its anticanonical map $\varphi: Y \to \mathbb P_2$ is branched along the proper transform $B$ of $C_\mathrm{sing}$ in $Y$ and its image $\varphi(B)$ is isomorphic to Klein's quartic curve. It follows that the surface $\{x_1^3x_2+x_2^3x_3+x_3^3x_1+x_4^4 =0\} \subset \mathbb P_3$ can also be characterized by the degeneration of the family $\mathcal M$ (cf.\ Theorem 1 in \cite{FKH1} and Theorem 5.4 in \cite{kf}). The blow-up map $ Y \to \mathbb P_2$ corresponds to the equivariant Mori reduction of $Y$ with respect to the subgroup $C_3 \ltimes C_7$ of $G$.
\end{remark}
%
%
\subsection{The group \texorpdfstring{$A_6$}{A6}}\label{A6Valentiner}
Let $G \cong A_6$ be the alternating group of degree 6 acting on a K3-surface $X$. Since $A_6$ is a simple group this action effective and symplectic. Let $\sigma$ be an antisymplectic involution on $X$ centralizing $G$ and assume that $\mathrm{Fix}_X(\sigma) \neq \emptyset$.  By Theorem \ref{roughclassi}, the K3-surface $X$ is a double cover of a Del Pezzo surface $Y$ with an effective action of $A_6$. 
\begin{lemma}
The Del Pezzo surface $Y$ is isomorphic to $\mathbb P_2$ with a uniquely determined action of $A_6$ given by the nontrivial central extension $V=3 .\ A_6$ of degree three known as \emph{Valentiner's group}. 
\end{lemma}
\begin{proof}
We go through the list of Del Pezzo surfaces according to their degree.

If $Y$ has degree one, then $| -K_Y|$ has precisely one base point which would have to be an $A_6$-fixed point. This is contrary to the fact that $A_6$ has no faithful two-dimensional representation. 

We recall that the stabilizer of a (-1)-curve $E$ in $Y$ is either cyclic or dihedral (Remark \ref{stab of minus one curve}). In particular, its order is at most 12 and therefore its index in $A_6$ is at least 30. This argument excludes Del Pezzo surfaces $Y$ of degree $\mathrm{deg}(Y) \in \{3,4,5,6,7\}$ since the number of (-1)-curves on $Y$ equals  27, 16, 10, 6, 3, respectively. Furthermore, the configuration of 56 exceptional curves on a Del Pezzo surface of degree 2 can neither be a single $A_6$-orbit nor the union of orbits and therefore  $\mathrm{deg}(Y) \neq 2$.

As was noted before the blow-up of $\mathbb P_2$ in one point is never $G$-minimal, hence
it remains to exclude $Y \cong \mathbb P_1 \times \mathbb P_1$. Assume there is an action of $A_6$ on $\mathbb P_1 \times \mathbb P_1$. Since $A_6$ has no subgroups of index two, it follows that $A_6 < \mathrm{PSL}_2(\mathbb C) \times  \mathrm{PSL}_2(\mathbb C)$ and both canonical projections are $A_6$-equivariant. This yields a contradiction since $A_6$ admits neither an effective action on $\mathbb P_1$ nor nontrivial normal subgroups of ineffectivity.

It follows that $Y \cong \mathbb P_2$. The action of $A_6$ on $\mathbb P_2$ is given by linear representation of a degree three central extension of $A_6$. Since $A_6$ has no faithful three-dimensional representation, this extension is nontrivial and isomorphic the unique nontrivial degree three extension $V=3.A_6$ known as Valentiner's group. Up to equivariant equivalence, there is a unique action of $A_6$ on $\mathbb P_2$. This follows from the classification of finite subgroup of $\mathrm{SL}_3(\mathbb C)$ (cf.\ \cite{blichfeldtbook}, \cite{blichfeldt}, and \cite{YauYu}) and can also be derived as follows: We need to show that any two actions induced by projective representations $\rho_i$ of $A_6$ are equivalent.
Restricting $\rho_i$ to the subgroup $A_5$ we obtain linear representation of $A_5$ and after a change of coordinates $\rho_1(A_5) = \rho_2(A_5) \subset \mathrm{SL}_3(\mathbb C)$. We fix a subgroup $A_4$ in $A_5$ and consider its normalizer $N$ in $A_6$. The groups $N$ and $A_5$ generate the full group $A_6$ and it suffices to prove that $\rho_1(N) = \rho_2(N)$. This is shown by considering an explicit three-dimensional representation of $A_4 < A_5$ and the normalizer $\mathcal N$ of $A_4$ inside $\mathrm{PSL}_3(\mathbb C)$. The group $A_4$ has index two in $\mathcal N$ and therefore $\mathcal N = \rho_1(N)= \rho_2(N)$.. 
\end{proof}
The covering $X \to Y$ is branched along an invariant curve $B$ of degree six. This curve is defined by an invariant polynomial $F_{A_6}$ of degree six, which is unique by Molien's formula. Its explicit equation is derived in \cite{crass}. In appropriately chosen coordinates,
\[
F_{A_6}(x_1,x_2,x_3) = 10 x_1^3x_2^3+ 9 x_1^5x_3 + 9 x_2^3x_3^3-45 x_1^2 x_2^2 x_3^2-135 x_1 x_2 x_3^4 + 27 x_3^6.
\]
The action of $A_6$ on $\mathbb P_2$ induces an action of a central degree two extension of $E$ on the double cover branched along $\{F_{A_6}=0\}$. 
By the same arguments as in the case $G = L_2(7)$ we may conclude that $E$ splits as $E_\mathrm{symp} \times C_2 = A_6 \times C_2$ where $C_2$ is generated by the antisymplectic covering involution. 
This proves the existence of a unique K3-surface with $A_6 \times C_2$-symmetry to which we refer as the \emph{Valentiner surface}. We have shown:
\begin{proposition}
Let $X$ be a K3-surface with a symplectic action of $A_6$ and let $\sigma$ be an antisymplectic involution on $X$ centralizing $A_6$ with $\mathrm{Fix}_X(\sigma) \neq \emptyset$. Then $X$ is equivariantly isomorphic to the Valentiner surface.
\end{proposition}
%
%
\subsection{The group \texorpdfstring{$S_5$}{S5}}
Let $X$ be a K3-surface with a symplectic action of $G = S_5$ and let $\sigma$ denote an antisymplectic involution centralizing $G$. We assume that $\mathrm{Fix}_X(\sigma) \neq \emptyset$. 
We may apply Theorem \ref{roughclassi} yielding that $X/ \sigma =Y$ is a $G$-minimal Del Pezzo surface and $\pi: X \to Y$ is branched along a smooth connected curve $B$ of genus $g(B) = 13-e(Y)$.
We show that only very few Del Pezzo surfaces admit an effective action of $S_5$ or a smooth $S_5$-invariant curve of appropriate genus.
\begin{lemma}
The degree $\mathrm{deg}(Y)=d$ of the Del Pezzo surface $Y$ is either three or five. 
\end{lemma}
\begin{proof}
We go through the list of $G$-minimal Del Pezzo surfaces and exclude all possible cases except $d \in \{3,5\}$.

Assume $Y \cong \mathbb P_2$, i.e., $G \hookrightarrow \mathrm{PSL}_3(\mathbb C)$, and let $\tilde G$ denote the preimage of $G$ in $\mathrm{SL}_3(\mathbb C)$. Since $A_5$ has no nontrivial central extension of degree three, it follows that the preimage of $A_5 < S_5$ in $\tilde G$ splits as $\tilde A_5 = A_5 \times C_3$. 
Let $g \in S_5$ be any transposition and pick $\tilde g$ in its preimage with $\tilde g^2 = \mathrm{id}$. Now $\tilde g$ and $ A_5$ generate a copy of $S_5$ in $\mathrm{SL}_3(\mathbb C)$. This is a contradiction since the irreducible representations of $S_5$ have dimensions $1,4,5$ or $6$ 

Next assume that $Y$ is isomorphic to $\mathbb P_1 \times \mathbb P_1$. We investigate the action of $S_5 = A_5 \rtimes C_2$ using the fact that $A_5$ is a simple group. Recalling $\mathrm{Aut}(Y) = (\mathrm{PSL}_2( \mathbb C) \times  \mathrm{PSL}_2(\mathbb C)) \rtimes C_2$ it follows that $A_5 < \mathrm{PSL}_2( \mathbb C) \times  \mathrm{PSL}_2( \mathbb C)$, i.e, the canonical projections onto the factors are $A_5$-equivariant.
If $A_5$ acts trivially on one of the factors, then it must act nontrivially on the second factor
and in particular the generator $\tau$ of the outer $C_2$ in $S_5$ stabilizes
the factors. It follows that $S_5$ 
acts effectively on the second factor. This is impossible since there is no effective action of $S_5$ on $\mathbb P_1$. We conclude that $A_5$ acts effectively on both factors and $\tau$ exchanges them.
We consider an element $\lambda$ of order five in $A_5$ and chose coordinates on $\mathbb P_1 \times \mathbb P_1$ such that $\lambda$ acts by
$([z_1:z_2],[w_1:w_2]) \mapsto ([\xi z_1:z_2],[\xi ^a w_1:w_2])$
for some $a \in \{1,2,3,4\}$ and $\xi^5 =1$. The automorphism $\lambda$ has four fixed points in $Y$.
Since it lifts to a symplectic automorphism on the K3-surface $X$ with four fixed points, all fixed points must lie on the branch curve.
The branch curve $B \subset Y$ is a smooth invariant curve linearly equivalent to $-2K_Y$ and is therefore given by an $S_5$-semi-invariant polynomial $f$ of bidegree $(4,4)$. Since $f$ must be invariant with respect to the commutator subgroup $S_5' = A_5$, it is a linear combination of $\lambda$-invariant monomials of bidegree $(4,4)$.
For each choice of $a$ one lists all $\lambda$-invariant monomials of bidegree $(4,4)$. In the case $a=1$ these are
$z_1 z_2^3 w_1^4,\,\, z_1^2 z_2^2 w_1^3 w_2 ,\,\, z_1^3 z_2 w_1^2 w_2^2,\,\,  z_1^4 w_1 w_2^3 , \,\,z_2^4 w_2^4$.
As $f$ must vanish at $p_1 \dots p_4$, one sees that it may not contain $z_2^4 w_2^4$. The remaining monomials have a common component $z_1 w_1$ such that $f$ factorizes and $B$ must be reducible, a contradiction.
The same argument can be carried out for each choice of $a$. It follows that the action of $S_5$ on  $\mathbb P_1 \times \mathbb P_1$ does not admit invariant irreducible curves of bidegree $(4,4)$. This eliminates the case $Y \cong \mathbb P_1 \times \mathbb P_1$.

Again using the fact that the largest subgroup of $S_5$ which can stabilize a (-1)-curve in $Y$ is the group $D_{12}$ of index 10,
it follows that the number of (-1)-curves in a $G$-orbit is at least 10
and Del Pezzo surfaces of degree four or six can be excluded by considering the action of $G$ on their set of (-1)-curves.

If $d =2$, then the anticanonical map defines an $\mathrm{Aut}(Y)$-equivariant double cover of $\mathbb P_2$. The induced action of $S_5$ on $\mathbb P_2$ would have to be effective and therefore we obtain a contradiction as in the case $Y \cong \mathbb P_2$.

If $d=1$ then the anticanonical system $|-K_Y|$ is known to have precisely one base point which has to be fixed point of the action of $S_5$. Since $S_5$ has no faithful two-dimensional representation, this is a contradiction. 
\end{proof}
\subsection*{Double covers of Del Pezzo surfaces of degree three}
The following example of a K3-surface $X$ with a symplectic action of $S_5$  can be found in Mukai's list \cite{mukai}. We identify an antisymplectic involution $\sigma$ on $X$ centralizing $S_5$ such that $X/\sigma$ is a Del Pezzo surface of degree three.
\begin{example}\label{MukaiS5}
Let $X$ be the K3-surface in $\mathbb P_5$ given by 
$\sum_{i=1}^5 x_i = \sum_{i=1}^6 x_1^2 = \sum_{i=1}^5 x_i^3=0$
 and let $S_5$ act on $\mathbb P_5$ by permuting the first five variables and by the character $\mathrm{sgn}$ on the last variable. This induces a symplectic action of $S_5$ action on $X$ (cf.\ \cite{mukai}).
The involution $\sigma : \mathbb P_5 \to \mathbb P_5$ given by $[x_1:x_2:x_3:x_4:x_5:x_6] \mapsto [x_1:x_2:x_3:x_4:x_5:-x_6]$ defines an antisymplectic involution on $X$ centralizing the action of $S_5$. 
The quotient $Y$ of $X$ by $\sigma$ is given by restricting then rational map $[x_1:x_2:x_3:x_4:x_5:x_6] \mapsto [x_1:x_2:x_3:x_4:x_5]$ to $X$. The surface $Y$ is given by 
$\{ \sum_{i=1}^5 y_i= \sum_{i=1}^5 y_i^3 =0\} \subset \mathbb P_4$
and is isomorphic to the Clebsch diagonal surface  
$\{z_1^2 z_2 + z_1 z_3^2 + z_3 z_4^2 + z_4 z_2^2 = 0\} \subset \mathbb P_3$ (cf.\ Theorem 10.3.10 in \cite{dolgachev}), a Del Pezzo surface of degree three. The branch set $B$ is given by $\{ \sum_{i=1}^5 y_1^2=0\} \cap Y \subset \mathbb P_4$.
By the following proposition, this example is the unique K3-surface with $S_5 \times \langle \sigma \rangle$-symmetry such that $X/\sigma$ is a Del Pezzo surface of degree three. 
\end{example}
\begin{proposition}\label{S5 on degree three}
Let $X$ be a K3-surface with a symplectic action of the group $S_5$ centralized by an antisymplectic involution $\sigma$. If $Y=X/\sigma$ is a Del Pezzo surface of degree three, then $X$ is equivariantly isomorphic to Mukai's $S_5$-example $\{\sum_{i=1}^5 x_i = \sum_{i=1}^6 x_1^2 = \sum_{i=1}^5 x_i^3=0\} \subset \mathbb P_5$. 
\end{proposition}
\begin{proof}
We consider the $\mathrm{Aut}(Y)$-equivariant embedding of the degree three Del Pezzo surface $Y$ into $\mathbb P_3$ given by the anticanonical map. Any automorphism of $Y$ is induced by an automorphism of the ambient projective space. 
It follows from the representation and invariant theory of the group $S_5$ that a Del Pezzo surface of degree three with an effective action of the group $S_5$ is equivariantly isomorphic the Clebsch cubic $\{z_1^2 z_2 + z_1 z_3^2 + z_3 z_4^2 + z_4 z_2^2 = 0\} \subset \mathbb P_3$ (cf.\ Theorems 10.3.9 and 10.3.10, Table 10.3 in \cite{dolgachev}). We show that the ramification curve $B \in |-2K_Y| $ is given by intersecting $Y$ with a quadric in $\mathbb P_3$. 

Applying the formula 
$h^0(Y, \mathcal{O}(-rK_Y))= 1+ \frac{1}{2}r(r+1) \cdot \mathrm{deg}(Y)$
(cf.\ e.g.\ Lemma 8.3.1 in \cite{dolgachev})
we obtain $h^0(Y,\mathcal{O}( -2K_Y))= 10$. This is also the dimension of the space of sections of $\mathcal O(2)$ in $\mathbb P_3$. It follows that the restriction map
$H^0(\mathbb P_3, \mathcal O (2)) \to H^0(Y, \mathcal O(-2K_Y))$
is surjective and $B = Y \cap Q$ for some quadric $Q = \{f=0\}$ in $\mathbb P_3$.

Since $B$ is an $S_5$-invariant curve in $Y$, it follows that for each $g \in S_5$ the intersection of $gQ = \{ f \circ g^{-1} =0\}$ and Y coincides with $B$. It follows that 
there exists a constant $c \in \mathbb C^*$ such that $(f \circ g^{-1}) - cf$ vanishes identically on $Y$. Since $Y$ is irreducible, this implies $f \circ g^{-1} = cf$. It follows that the polynomial $f$ is an $S_5$-semi-invariant and therefore invariant with respect to the commutator subgroup $A_5$.

We have previously noted that after a suitable linear change of coordinates the surface $Y$ is given by $\{ \sum_{i=1}^5 y_i= \sum_{i=1}^5 y_i^3 =0\} \subset \mathbb P_4$ where $S_5$ acts  by permutation. 
In these coordinates the semi-invariant polynomial $f$ is given by
$a \sum_{i=1}^5 y_i^2 + b(\sum_{i=1}^5 y_i)^2 =0$
for some $a,b \in \mathbb C$.
In particular,  $B = Y \cap \{f=0\} = Y \cap  \{\sum_{i=1}^5 y_i^2 =0\}$
and 
$X$ is Mukai's $S_5$-example discussed in Example \ref{MukaiS5}.
\end{proof}
\subsection*{Double covers of Del Pezzo surfaces of degree five}
A second class of candidates of K3-surfaces with $S_5 \times C_2$-symmetry is given by double covers of Del Pezzo surfaces of degree five. 
\begin{remark}
Any two Del Pezzo surfaces of degree five are isomorphic and the automorphisms group of a Del Pezzo surface $Y$ of degree five is $S_5$. The ten (-1)-curves on $Y$ form a graph known as the \emph{Petersen graph}. This graph has $S_5$-symmetry and every symmetry of the abstract graph is induced by a unique automorphism of the Del Pezzo surface. 
\end{remark}
\begin{proposition}\label{S5 on degree five}
Let $X$ be a K3-surface with a symplectic action of the group $S_5$ centralized by an antisymplectic involution $\sigma$. If $Y=X/\sigma$ is a Del Pezzo surface of degree five, then $X$ is equivariantly isomorphic to the minimal desingularization of the double cover of $\mathbb P_2$ branched along the sextic
\begin{align*}
 C_{S_5}=\{&2x_1^4x_2x_3+2x_1x_2^4x_3+2x_1x_2x_3^4-2x_1^4x_2^2-2x_1^4x_3^2-2x_1^2x_2^4-2x_1^2x_3^4\\
	   &-2x_2^4x_3^2-2x_2^2x_3^4+2x_1^3x_2^3+2x_1^3x_3^3+2x_2^3x_3^3+x_1^3x_2^2x_3+x_1^3x_2x_3^2\\
	   &+x_1^2x_2^3x_3+x_1^2x_2x_3^3+x_1x_2^3x_3^2+x_1x_2^2x_3^3-6x_1^2x_2^2x_3^2=0\}.
\end{align*}
\end{proposition}
\begin{proof}
Let $B \subset Y$ denote the branch locus of the covering $X \to Y$. The curve $B$ is invariant with respect to the full automorphism group of $Y$.
The Del Pezzo surface $Y$ is the blow-up $b: Y \to \mathbb P_2$ of four points $p_1,p_2,p_3,p_4 \in \mathbb P_2$ in general position. We may choose coordinates $[x_1:x_2:x_3]$ on $\mathbb P_2$ such that $p_1=[1:0:0],\ p_2=[0:1:0],\ p_3=[0:0:1],\ p_4=[1:1:1]$,
and consider the $S_4$-action on $\mathbb P_2$ permuting the points $\{p_i\}$. The isotropy at each point $p_i$ is isomorphic to $S_3$ and induces an effective $S_3$-action on $E_i =b^{-1}(p_i)$. 

Let $E$ be any (-1)-curve on $Y$. By adjunction $E\cdot B=2$. Since $Y$ contains precisely ten (-1)-curves forming an $S_5$-orbit, the group $H = \mathrm{Stab}_{S_5}(E)$ has order 12 and all stabilizer groups of (-1)-curves in $Y$ are conjugate.
It follows that the group $H$ contains $S_3$, which is acting effectively on $E$, and therefore $H$ is isomorphic to the dihedral group of order 12. The points of intersection $B\cap E$  form an $H$-invariant subset of $E$. Since $H$ has no fixed points in $E$ and precisely one orbit $H.p = \{p,q\}$ consisting of two elements, it follows that $B$ meets $E$ transversally in $p$ and $q$.
In particular, each curve $E_i$
meets $B$ in two points and the image $C = b(B)$ is a sextic curve with four nodes $\{p_1, \dots, p_4\}$. 
It is invariant with respect to the action of $S_4$ given by permutation on $p_1, \dots p_4$. For simplicity, we first only consider the action of $S_3$ permuting $p_1,p_2,p_3$ and conclude that $C$ is given by $\{f=\sum a_i f_i =0 \}$ for
\begin{align*}
f_1&=x_1^6+x_2^6 +x_3^6, &f_5=&x_1^3x_2^3+x_1^3x_3^3+x_2^3x_3^3, \\
f_2&=x_1^5x_2 + x_1^5x_3+ x_1x_2^5 +x_1x_3^5+x_2^5x_3+x_2x_3^6, &f_6=&x_1^3x_2^2x_3+x_1^3x_2x_3^2+x_1^2x_2^3x_3 \\ 
f_3&=x_1^4x_2x_3+x_1x_2^4x_3+x_1x_2x_3^4,  &&+x_1^2x_2x_3^3+x_1x_2^3x_3^2+x_1x_2^2x_3^3,\\
f_4&=x_1^4x_2^2+x_1^4x_3^2+x_1^2x_2^4+x_1^2x_3^4+x_2^4x_3^2+x_2^2x_3^4, 
&f_7=&x_1^2x_2^2x_3^2.
\end{align*}
The fact that $C$ passes through $p_i$ and is singular at $p_i$ yields $a_1=a_2=0$ and
$3a_3+6a_4+3a_5+6a_6+a_7=0$.
The two tangent lines of $C$ at the node $p_i$ correspond to the unique $\mathrm{Stab}(E_i)$-orbit of length two in $E_i$. We consider the point $p_3$ and the subgroup $S_3 < S_4$ stabilizing $p_3$. The action of $S_3$ on $E_3$ is given by the linearized $S_3$-action on the set of lines through $p_3$. One checks that in local affine coordinates $(x_1,x_2)$ the unique orbit of length two corresponds to the line pair $x_1^2 -x_1 x_2+x_2^2 =0$.
Dehomogenizing $f$ at $p_3$, i.e., setting $x_3=1$, we obtain the local equation $f_\mathrm{dehom}$ of $C$ at $p_3$. 
The polynomial $f_\mathrm{dehom}$ modulo terms of order three or higher must be a multiple of $x_1^2 -x_1 x_2+x_2^2 =0$. Therefore $a_3 = -a_4$.

Next we consider the intersection of $C$ with the line $L_{34} = \{x_1=x_2\}$ joining $p_3$ and $p_4$. We know that $f|_{L_{34}}$ vanishes of order two at $p_3$ and $p_4$ and at one or two further points on $L_{34}$. 
Let $\widetilde L_{34}$ denote the proper transform of $L_{34}$ inside the Del Pezzo surface $Y$. The curve $\widetilde L_{34}$ is a (-1)-curve, hence its stabilizer $\mathrm{Stab}_G(\widetilde L_{34})$ is isomorphic to $D_{12} = S_3 \times C_2$. The factor $C_2$ acts trivially on $\widetilde L_{34}$. Since the intersection of $\widetilde L_{34}$ with $B$ is $\mathrm{Stab}_G(\widetilde L_{34})$ invariant, it follows that $\widetilde L_{34} \cap B$ is the unique $S_3$-orbit a length two in $\widetilde L_{34}$.  
We wish to transfer our determination of the unique $S_3$-orbit of length two in $E_3$ above to the curve $\widetilde L_{34}$ using an automorphism of $Y$ mapping $E_3$ to $\widetilde L_{34}$. Consider the automorphism $\varphi$ of $Y$ induced by the birational map of $\mathbb P_2$ given by $[x_1:x_2:x_3] \mapsto [x_1(x_3-x_2):x_3(x_1-x_2):x_1x_3]$
(cf.\ Theorem 10.2.2 in \cite{dolgachev}) and let $\psi$ be the automorphism of $Y$ induced by the permutation of the points $p_2$ and $p_3$ in $\mathbb P_2$. Then $\psi \circ \varphi$ is an automorphism of $Y$ mapping $E_3$ to $\widetilde L_{34}$. If $[X_1:X_2]$ denote homogeneous coordinates on $E_3$ induced by the affine coordinates $(x_1,x_2)$ in a neighbourhood of $p_3$, then a point $[X_1:X_2] \in E_3$ is mapped to the point corresponding to $[X_1:X_1: X_1-X_2] \in L_{34} \subset \mathbb P_2$. It was derived above that the unique $S_3$-orbit of length two in $E_3$ is given by $X_1^2-X_1X_2 +X_2^2$ and it follows that the unique $S_3$-orbit of length two in $\widetilde L_{34}$ corresponds to the points $[x_1:x_1:x_3] \in \mathbb P_2$ fulfilling $x_1^2 -x_1x_3 +x_3^2 =0$. 
Therefore, $f|_{L_{34}}$ is a multiple of polynomial given by $x_1^2(x_1-x_3)^2(x_1^2 -x_1x_3 +x_3^2)$.
Comparing coefficients with $f(x_1:x_1:x_3)$ yields
\begin{align*}
2a_3+2a_6 &= 2a_5 +2a_6, & 2a_4 + a_5&= 2a_4+a_3,\\
8a_4+4a_5 &= 2a_4+2a_6 +a_7, &-6a_4 -3a_5 &= 2a_5 +2a_6.
\end{align*}
We conclude $a_3=a_5= -a_4=2$, $a_6=1$, and $a_7=-6$. So if $X$ as in the proposition exists, it is the double cover of $Y$ branched along the proper transform of $\{f = 2f_3 -2f_4+2f_5+f_6-6f_7  = 0\}$ in $Y$.

In order to prove existence, let $X$ be the minimal desingularization of the double cover of $\mathbb P_2$ branched along $\{f=0\}$. Then $X$ is the double cover of the Del Pezzo surface $Y$ of degree five branched along the proper transform $D$ of $\{f=0\}$ in $Y$.
Since all automorphisms of $Y$ are induced by explicit biholomorphic or birational transformation of $\mathbb P_2$ 
one can check by direct computations that $D$ is in fact invariant with respect to the action of $\mathrm{Aut}(Y) = S_5$.

On $X$ there is an action of a central extension $E$ of $S_5$, $\{\mathrm{id}\} \to C_2 \to E \to S_5 \to \{\mathrm{id}\}$. Let $E_\mathrm{symp}$ be the subgroup of symplectic automorphisms in $E$. Since $E$ contains the antisymplectic covering involution $E_\mathrm{symp} \neq E$. The image $N$ of $E_\mathrm{symp}$ in $S_5$ is normal and therefore either $N \cong S_5$ or $N \cong A_5$.
If $N \cong A_5$ and $| E_\mathrm{symp}| =60$, then  $E_\mathrm{symp} \cong A_5$. Lifting any transposition from $S_5$ to an element $g$ of order two in $E$, the group generated by $g$ and $E_\mathrm{symp}$ inside $E$ is isomorphic to $S_5$. It follows that $E$ splits as $S_5 \times C_2$ and $E / E_\mathrm{symp} \cong C_2 \times C_2$. Since this is not cyclic, we obtain a contradiction.  
If $N \cong A_5$ and $| E_\mathrm{symp}| =120$, then $ E = E_\mathrm{symp} \times C_2$, where the outer $C_2$ is generated by the antisymplectic covering involution $\sigma$, and $E / C_2 = S_5$ implies that $E_\mathrm{symp} \cong S_5$. This is contradictory to the assumption $N \cong A_5$.
In the last remaining case $N \cong S_5$. Since $E_\mathrm{symp} \neq E$, also $E_\mathrm{symp} \cong S_5$ and $E$ splits as $E_\mathrm{symp} \times C_2$. 
It follows that the action of $S_5$ on $Y$ induces an symplectic action of $S_5$ on the double cover $X$ centralized by the antisymplectic covering involution. This completes the proof of the proposition.
\end{proof}
We summarize Propositions \ref{S5 on degree three} and \ref{S5 on degree five} as follows.
\begin{proposition}
Let $X$ be a K3-surface with a symplectic action of the group $S_5$ centralized by an antisymplectic involution $\sigma$ with $\mathrm{Fix}_X(\sigma) \neq \emptyset$. Then $X$ is equivariantly isomorphic to either Mukai's $S_5$-example or the minimal desingularization of the double cover of $\mathbb P_2$ branched along the sextic $C_{S_5}$.
\end{proposition}
%
%
\subsection{The group \texorpdfstring{$M_{20} = C_2^4 \rtimes A_5$}{M20}}
We assume that a K3-surface $X$ with a symplectic action of $M_{20}$ centralized by an antisymplectic involution $\sigma$ and $\mathrm{Fix}_X(\sigma) \neq \emptyset$ exists. Applying Theorem \ref{roughclassi} we see that $X \to Y$ is branched along a single $M_{20}$-invariant smooth curve $B$ with $g(B) \geq 3$ on the Del Pezzo surface $Y$. By Hurwitz's formula, $| \mathrm{Aut} (B) | \leq 84(g(B)-1)$,
the genus of $B$ must be at least twelve. As $B$ is linearly equivalent to $-2K_Y$, the adjunction formula 
$2g(B)-2 = (K_Y,B) + B^2 = 2K_Y^2 $
implies $\mathrm{deg}(Y) = K_Y^2 \geq 11$. Since the degree of a Del Pezzo surface is at most nine this is a contradiction yielding the following non-existence result. 
\begin{proposition}
There does not exist a K3-surface with a symplectic action of $M_{20}$ centralized by an antisymplectic involution $\sigma$ with $\mathrm{Fix}_X(\sigma) \neq \emptyset$.
\end{proposition}
%
\subsection{The group \texorpdfstring{$F_{384} = C_2^4 \rtimes S_4$}{F384}}
Before we prove non-existence of K3-surfaces with $F_{384} \times C_2$-sym\-me\-try, we note the following useful fact about $S_4$-actions on Riemann surfaces.
\begin{remark}\label{S4 not on g=1,2}
The group $S_4$ does not admit an effective action on a Riemann surface of genus one or two.
The first case can be excluded by using the explicit shape of the automorphism group of a Riemann surface $T$ of genus one, $\mathrm{Aut}(T)= L \ltimes T$ for $L \in \{C_2, C_4, C_6\}$. 
The second case is excluded by considering the quotient by the hyperelliptic involution branched at six points.
Since $S_4$ has no normal subgroup of order two, the induced action of $S_4$ on the quotient $\mathbb P_1$ is effective and therefore has precisely one orbit consisting of six points. 
The isotropy group at the corresponding points in the hyperelliptic curve would be isomorphic to $C_4 \times C_2$, which is contradiction.
\end{remark}
\begin{proposition}
There does not exists a K3-surface with a symplectic action of $F_{384}$ centralized by an antisymplectic involution $\sigma$ with $\mathrm{Fix}_X(\sigma) \neq \emptyset$.
\end{proposition}
\begin{proof}
As above, assume that a K3-surface $X$ with these properties exists and apply Theorem \ref{roughclassi} to see that $X \to Y$ is branched along a single $F_{384}$-invariant smooth curve $B$ on the Del Pezzo surface $Y$. By Hurwitz's formula the genus of $B$ is at least 6.
It now follows from adjunction, $K_Y^2 = g(B)-1$, that the degree of the Del Pezzo surface $Y$ is at least five.
We consider the action of $F_{384}$ on the configuration of (-1)-curves on $Y$ and recall that the order of a stabilizer of a (-1)-curve in $Y$ is at most twelve (cf.\ Remark \ref{stab of minus one curve}) and therefore has index greater than or equal to $32$ in $G$. It follows that $Y$ is either $\mathbb P_1 \times \mathbb P_1$ or $\mathbb P_2$. In the first case, the canonical projections of $\mathbb P_1 \times \mathbb P_1$ are equivariant with respect to a subgroup of index two in $F_{384}$ and thereby contradict Lemma \ref{conicbundle}. Consequently, $Y \cong \mathbb P_2$
and $g(B) =10$.

We use the realization of $F_{384}$ as a semi-direct product $C_4^2 \rtimes S_4$ (cf.\ \cite{mukai}) and consider the quotient $Q$ of the branch curve $B$ by the normal subgroup $N = C_4^2$. On $Q$ there is the induced action of $S_4$. It follows from the remark above that $Q$ is either rational or $g(Q) >2$. In the second case the Riemann-Hurwitz formula applied to the covering $ B \to Q$,
\[
-18 = e(B) = 16 e(Q) - \text{branch point contributions} \leq -64,
\]
yields a contradiction.
It follows that $Q$ is a rational curve and that the branch point contribution of the covering $B \to Q$ is 50. Since isotropy groups at points in $B$ must be cyclic, the only possible isotropy subgroups of $N = C_4^2$ at points in $B$ are $C_2$ and $C_4$ and have index four or eight. The full branch point contribution must therefore be a multiple of four. This contradiction yields the non-existence claimed. 
\end{proof}
%
\subsection{The group \texorpdfstring{$A_{4,4} = C_2^4 \rtimes A_{3,3}$}{A4,4}}
By $S_{p,q}$ for $p+q =n$ we denote a subgroup $S_p \times S_q$ of $S_n$ preserving a partition of the set $\{1,\dots, n\}$ into two subsets of cardinality $p$ and $q$. The intersection of $S_{p,q}$ with $A_n$ is denoted by $A_{p,q}$.
\begin{proposition}
There does not exists a K3-surface with a symplectic action of $A_{4,4}$ centralized by an antisymplectic involution $\sigma$ with $\mathrm{Fix}_X(\sigma) \neq \emptyset$.
\end{proposition}
\begin{proof}
We again assume that a K3-surface with these properties exists and applying Theorem \ref{roughclassi} we see that $X \to Y$ is branched along a single $A_{4,4}$-invariant smooth curve $B$ on the Del Pezzo surface $Y$. The group $A_{4,4}$ is a semi-direct product $C_2^4 \rtimes A_{3,3}$ (see e.g.\ \cite{mukai}). We consider the quotient $Q$ of $B$ by the normal subgroup $N \cong C_2^4$. On $Q$ there is an action of $A_{3,3}$.
Since $A_{3,3}$ contains the subgroup $C_3 \times C_3$,
it follows that $Q$ not rational. In particular, $e(Q) \leq 0$. We apply the Riemann-Hurwitz formula to the covering $B \to Q$,
\[
2-2g(B) = e(B) = 16 e(Q) - \text{branch point contributions} \leq  - \text{branch point contributions}.
\]
As above, isotropy groups must be cyclic and the maximal possible isotropy group of the $C_2^4$-action on $B$ is $C_2$, which has index eight in $C_2^4$. Consequently, the branch point contribution at each branch point is eight.
The action of $C_3 \times C_3 < A_{3,3}$ on $Q$ has orbits of length greater than or equal to three. Therefore, the total branch point contribution must be greater than or equal to $24$. In particular, $g(B) = \mathrm{deg}(Y) +1 \geq 13$ contrary to $\mathrm{deg}(Y) \leq 9$.
\end{proof}
%
\subsection{The groups \texorpdfstring{$T_{192} = (Q_8 * Q_8) \rtimes S_3$}{T192} and \texorpdfstring{$H_{192} = C_2^4  \rtimes D_{12}$}{H192}}
By $Q_8$ we denote the quaternion group $\{\pm 1, \pm I, \pm J, \pm K\}$ where $I^2= J^2= K^2 = IJK = -1$. 
The central product $ Q_8 * Q_8 $ is defined as the quotient of $Q_8  \times Q_8$ by the central involution $(-1, -1)$.
Most importantly, we benefit from the observation that
both groups $T_{192}$ and $H_{192}$ are semi-direct products $C_2^3 \rtimes S_4$ (cf.\ \cite {mukai}).

We prove the following non-existence result. 
\begin{proposition}\label{192}
For $G =T_{192}$ or $G = H_{192}$ there does not exists a K3-surface with a symplectic action of $G$ centralized by an antisymplectic involution $\sigma$ with $\mathrm{Fix}_X(\sigma) \neq \emptyset$.
\end{proposition}
Assume that a K3-surface as in the proposition exists. Applying Theorem \ref{roughclassi} we see that $X \to Y$ is branched along a single $G$-invariant smooth curve $C$ on the Del Pezzo surface $Y$. The genus of $C$ is at least four by Hurwitz's formula and therefore $d = \mathrm{deg}(Y) \geq 3$.
%
\subsection*{The cases \texorpdfstring{$3 \leq  d \leq 7$}{d=3,4,5,6,7}}
We consider the action of $G$ on the Del Pezzo surface $Y$ of degree $\geq 3$, in particular the induced action on its configuration of (-1)-curves. By Remark \ref{stab of minus one curve} 
the stabilizer of a (-1)-curve in $Y$ has index $\geq 16$ in $G$ and we may immediately exclude the cases $\mathrm{deg}(Y) = 3,5,6,7$.
The automorphism group of a Del Pezzo surface of degree four is $C_2^4 \rtimes \Gamma$ for $\Gamma \in \{C_2, C_4, S_3, D_{10}\}$ (cf.\ Section 10.2.2 in \cite{dolgachev}). In particular, the maximal possible order is 160 and therefore $\mathrm{deg}(Y) \neq 4$.
It remains to consider the cases $Y \cong \mathbb P_2$ and $Y \cong \mathbb P_1 \times \mathbb P_1$
%
\subsection*{The case \texorpdfstring{$Y \cong \mathbb P_2$}{Y=P2}}
If $Y \cong \mathbb P_2$ than $g(C) =10$. 
We consider the quotient $Q$ of $C$ by the normal subgroup $N = C_2^3$ and obtain
\[
 -18 = e(C) = 8\cdot e(Q) - \text{branch point contributions}.
\]
By Lemma \ref{S4 not on g=1,2} 
the quotient $Q$ is either rational or $g(Q) >2$, i.e., $e(Q) \leq -4$. Since the second case is impossible it follows that $Q$ is a rational curve and the
branch point contribution must be $34$. The maximal possible isotropy subgroup of $N = C_2^3$ at a point in $C$ is $C_2$ and the full branch point contribution must be divisible by four. This is a contradiction.
%
\subsection*{The case \texorpdfstring{$Y \cong \mathbb P_1 \times \mathbb P_1$}{Y = P1 times P1}}
Assume that $Y \cong \mathbb P_1 \times \mathbb P_1$. The canonical projection $\pi_{1,2}: Y \to \mathbb P_1$ is equivariant with respect to a subgroup $\tilde G$ of $G$ of index at most two. It follows that $\tilde G$ fits into the exact sequences
$ \{\mathrm{id}\} \to I_i \to \tilde G \overset{(\pi_i)_*}{\to} H_i\to  \{\mathrm{id}\}$, 
where $I_i \cong C_2 \times C_2$ is the ineffectivity of the induced $H$-action on the base $\pi_i(Y)$ and $H_i \cong S_4$ (cf.\ proof of Lemma \ref{conicbundle}).
We consider the realization of $G$ as a semi-direct product $C_2^3 \rtimes S_4 = N \rtimes S_4$ and denote by $A = \tilde G \cap N$ the intersection of the normal subgroup $N$ with the index two normal subgroup $\tilde G$.
\begin{lemma}
The group $A$ is isomorphic to $C_2 \times C_2$.
\end{lemma}
\begin{proof}
It is sufficient to exclude the case $A=N$. We argue by contradiction and assume $A=N$. The commutator subgroup of $S_4$ is $A_4$ and therefore $A_4 < G' < \tilde G < G$. Since $\tilde G$ is known to be of order 96, it follows that $\tilde G = N \rtimes A_4$. We consider the isotropy group $I_1 \cong C_2 \times C_2 < \tilde G$ and its intersection with $A=N$.

If $A\cap I_1 = \{ \mathrm{id} \}$, then $A =N \cong C_2 ^3$ acts effectively on $\pi_1 (Y) =\mathbb P_1$, a contradiction. If $A \cap I_1 = I_1$, i.e., if $I_1$ is contained in $A$, then the quotient group $\tilde G / I_1$, which acts effectively on $\mathbb P_1$ by definition, is isomorphic to $C_2 \rtimes A_4 \cong C_2 \times A_4$. This group does however not admit an effective action on $\mathbb P_1$.

It remains to consider the case $A \cap I_1 \cong C_2$. 
The only nontrivial normal subgroup of $A_4$ is isomorphic to $C_2 \times C_2$. Since $I_1 \cong C_2 \times C_2$ is not contained in $A_4$ in the case under consideration, it follows that $I_1 \cap A_4 = \{ \mathrm{id}\}$.
Written as a subgroup of the semi-direct product $\tilde G = N \rtimes A_4$ the group $I_1 \cong C_2 \times C_2$ is of the form
$I_1 = \{ (\mathrm{id},\mathrm{id} ), (\sigma_1,\mathrm{id}), (\sigma_2, a), (\sigma_3, a)\}$
for $\sigma_i \in N$ and $a \in A_4$.
Using the fact that $I_1$ is a normal subgroup of $\tilde G$ one finds that $g a g^{-1} = a$ for each $g \in A_4 < \tilde G$, a contradiction.
\end{proof}
By the lemma above, the set $N\backslash A$ consists of four elements. These are involutions in $G$ not respecting the product structure of $Y \cong \mathbb P_1 \times \mathbb P_1$. In particular, each element of $N\backslash A$ exchanges the factors of $Y$.
Since both $N$ and $A$ are normal subgroup of $G$, it follows that $G$ acts on $N$, $A$, and also $N\backslash A$ by conjugation.
We consider an element $\lambda$ of order three contained in $A_4$, the commutator of the chosen copy of $S_4$ inside $G$. In particular, 
$\lambda \in A_4 = S_4' < G' < \tilde G < G$.
The action of $\lambda$ on the set $N\backslash A$ has at least one fixed point, i.e., there exists an element $\sigma \in N\backslash A$ such that $\lambda \sigma \lambda^{-1} = \sigma$.

After a suitable change of coordinates, the action of $\sigma$ on $Y$ is of the form $\sigma (z,w) = (w,z)$. Since $\lambda \in \tilde G$ respects the product structure and commutes with $\sigma$, it follows that
$\lambda (z,w) = (\lambda_z z, \lambda _z w)$.
for some $\lambda_z \in \mathrm{PSL}_2(\mathbb C)$ of order three. We choose homogeneous coordinates on $\pi_i(Y) \cong \mathbb P_1$ such that
$\lambda (z,w) = \lambda ([z_0: z_1], [w_0: w_1]) = ([ \xi z_0: z_1], [\xi w_0: w_1])$,
where $\xi$ is some nontrivial third root of unity. Note that this choice of coordinates does not effect the chosen shape of $\sigma$.

The branch curve $B$ of the covering $\pi: X \to Y$ is given by a $G$-(semi)-invariant polynomial $f$ of bidegree $(4,4)$. Since $\lambda$ is contained in the commutator of $G$, the polynomial $f$ is invariant with respect to the induced action of $\lambda$. In particular, it is a linear combination of $\lambda$-invariant monomials of bidegree $(4,4)$. These are
\[
z_0^4 w_0^2 w_1^2, \, \,
z_0^3 z_1 w_0^3 w_1,\, \,
z_0^3 z_1 w_1^4,\, \,
z_0^2 z_1^2 w_0^4,\, \,
z_0^2 z_1^2 w_0 w_1^3,\, \,
z_0 z_1^3 w_0^2 w_1^2,\, \,
z_1^4 w_0^3 w_1,\, \,
z_1^4 w_1^4.
\]
It follows that three of the four $\lambda$-fixed points on $Y$, namely 
$([1:0],[1:0])$, $([0:1][1:0])$, and  $([1:0],[0:1])$ lie on $B$ and the action of $\lambda$ on the double covering $X$ has at most five fixed points. However by Remark \ref{order of symp aut}, a symplectic action of $C_3$ on the K3-surface $X$ has precisely six fixed points. This yields a contradiction and we conclude $Y \not \cong \mathbb P_1 \times \mathbb P_1$.

As we have obtained contradictions for all possible choices of Del Pezzo surfaces $Y$ the non-existence claimed in Proposition \ref{192} follows. 
%
%
\subsection{The group \texorpdfstring{$N_{72} = C_3^2 \rtimes D_8$}{N72}}
We let $X$ be a K3-surface with a symplectic action of $G=N_{72}$ centralized by an antisymplectic involution $\sigma$ with $\mathrm{Fix}_X(\sigma) \neq \emptyset$.
Note that in this case we may not apply Theorem \ref{roughclassi} and therefore begin by excluding that a $G$-minimal model of $Y=X/\sigma$ is an equivariant conic bundle. Then, after limiting the possible degrees of $G$-minimal models $Y_\mathrm{min}$, we study the possibility of rational branch curves.
\begin{lemma}\label{N72notconicbundle}
Any $G$-minimal model of $Y$ is a Del Pezzo surface.
\end{lemma}
\begin{proof}
Assume the contrary and let $Y_\mathrm{min}$ be an equivariant conic bundle and a $G$-minimal model of $Y$.  We consider the induced action of $G$ on the base $\mathbb P_1$
and denote by $I \lhd G$ the ineffectivity of this action. Arguing as in the proof of Lemma \ref{conicbundle}, we see that $I$ is trivial or isomorphic to either $C_2$ or $C_2 \times C_2$. In all cases the quotient $G/I$ contains the subgroup $C_3 \times C_3$, which has no effective action on $\mathbb{P}_1$.
\end{proof}
\begin{lemma}
Let $Y_\mathrm{min}$ be a Del Pezzo surface and a $G$-minimal model of $Y$. Then $\mathrm{deg}(Y_\mathrm{min}) \leq 4$.
\end{lemma}
\begin{proof}
We exclude all $G$-minimal Del Pezzo surfaces of degree at least five:

A Del Pezzo surface of degree five has automorphism group $S_5$ (cf.\ Theorem 10.2.2 in \cite{dolgachev}). Since $N_{72} \nless S_5$, it follows that $\mathrm{deg}(Y_\mathrm{min}) \neq 5$

The automorphism group of a Del Pezzo surface of degree six is $(\mathbb C^* )^2 \rtimes (S_3 \times C_2)$ (cf.\ Theorem 10.2.1 in \cite{dolgachev}). Assume that $N_{72} = C_3^2 \rtimes D_8$ is contained in this group and consider the intersection $ A = N_{72} \cap (\mathbb C^* )^2$. The quotient of $N_{72}$ by $A$ is a subgroup of $S_3 \times C_2$ and may not contain a copy of $C_3^2$. Therefore, the order of $A$ is at least six and $A$ contains a copy of $C_3$. If $|A| =6$, then $A = C_6 = C_3 \times C_2$ and $C_2$ is central in $N_{72}$. Using explicitly the group structure of $N_{72}$ one finds however that there is no copy of $C_2$ in $N_{72}$ centralizing $C_3 \times C_3$.
If $|A|  >6$, then the centralizer of $C_3$ in $D_8$ has order greater than 2. This is contrary to the fact that for every choice of $C_3$ inside $C_3 \times C_3$ the centralizer inside $D_8$ is isomorphic to $C_2$. It follows that $\mathrm{deg}(Y_\mathrm{min}) \neq 6$.

If $G$ acts on $\mathbb P_1 \times \mathbb P_1$, then the canonical projections are equivariant with respect to a subgroup $H$ of index two in $G$. We consider one of these projections. The action of $H$ induces an effective action of $H/I$ on the base $\mathbb P_1$, where the group $I$ is either trivial or isomorphic to $C_2$ or $C_2 \times C_2$. In all case we find an effective action of $C_3 ^2$ on the base, a contradiction.  

It remains to exclude $\mathbb P_2$.
If $N_{72}$ acts on $\mathbb P_2$ we consider its embedding into $\mathrm{PSL}_3(\mathbb C)$, in particular the realization of the subgroup $C_3^2= \langle a \rangle \times \langle b \rangle$ 
and its lifting to $\mathrm{SL}_3(\mathbb C)$. 
One uses explicit realizations of the generators $a$ and $b$ in appropriately chosen coordinates and checks that the action of $D_8$ on $C_3^2$ cannot be realized in $\mathrm{PSL}_3(\mathbb C)$. The calculation is omitted here and the reader is referred to Section 4.8 in \cite{kf} for details.
It follows that there is no action of $N_{72}$ on $\mathbb P_2$.
\end{proof}
\begin{lemma}\label{noratN72}
There are no rational curves in $\mathrm{Fix}_X(\sigma)$.
\end{lemma}
\begin{proof}
 Let $n$ denote the total number of rational curves in $\mathrm{Fix}_X(\sigma)$ and recall $n \leq 10$. If $n \neq 0$, let $C$ be a rational curve in the image of  $\mathrm{Fix}_X(\sigma)$ in $Y$ and let $H = \mathrm{Stab}_G(C)$ be its stabilizer. 
Its order is at least eight and its action on $C$ is effective. 
First note that $G$ does not contain $S_4 = O_{24}$ as a subgroup. If this were the case, consider the intersection $S_4 \cap C_3^2$ and the quotient $S_4 \to S_4 / (S_4 \cap C_3^2) < D_8$. Since the only nontrivial normal subgroups of $S_4$ are $A_4$ and $C_2 \times C_2$, this leads to a contradiction.
Consequently, the order of $H$ is at most twelve. In particular, $n \geq 6$. We will obtain a contradiction by showing $n \leq 4$.

Since $| H| \geq 8$ and $C_8 \nless G$, the group $H$ is not cyclic and any $H$-orbit on $C$ consists of at least two points.
It follows from $C^2 = -4$ that $C$ must meet the union of Mori fibers and the union of Mori fibers meets the curve $C$ in at least two points. Recalling that each Mori fibers meets the branch locus $B$ in at most two points we see that at least $n$ Mori fibers meeting $B$ are required. However, no configuration of $n$ Mori fibers is sufficient to transform the curve $C$ into a curve on a Del Pezzo surface and further Mori fibers are required. By invariance, the total number $m$ of Mori fibers must be at least $2n$.
Combining the Euler-characteristic formula \eqref{eulerchar}
with our observation $\mathrm{deg}(Y_\mathrm{min}) \leq 4$, i.e., $e(Y_\mathrm{min}) \geq 8$ we see that $n \leq 4$.
\end{proof}
\begin{proposition}
Let $X$ be a K3-surface with a symplectic action of $G=N_{72}$ centralized by an antisymplectic involution $\sigma$ with $\mathrm{Fix}_X(\sigma) \neq \emptyset$.
 Then the quotient surface $Y = X / \sigma $ is $G$-minimal and isomorphic to the Fermat cubic $\{x_1^3 + x_2^3 +x_3^3 +x_4^3 =0\} \subset \mathbb P_3$. Up to equivalence, there is a unique action of $G$ on $Y$ and the branch locus of $X \to Y$ is given by $\{x_1x_2 + x_3x_4 =0\}$. In particular, $X$ is equivariantly isomorphic to Mukai's $N_{72}$-example.
\end{proposition}
\begin{proof}
Using the Euler-cha\-rac\-te\-ris\-tic formula \eqref{eulerchar} together with the two previous lemmata, we find that the total number $m$ of Mori fibers is bounded by four. Since the maximal order of a stabilizer group of a Mori fiber is twelve (cf.\ proof of Theorem \ref{roughclassi}) we conclude $m=0$, i.e., $Y$ must be $G$-minimal.

In order to conclude that $Y$ is the Fermat cubic we consult Dolgachev's lists of automorphisms groups of Del Pezzo surfaces of degree less than or equal to four (\cite{dolgachev} Section 10.2.2; Tables 10.3; 10,4; and 10.5):
It follows immediately from the order of $G$ that $Y$ is not of degree two or four. If $G$ were a subgroup of an automorphism group of a Del Pezzo surface of degree one, it would contain a central copy of $C_3$, which is not the case.
After excluding the cases $\mathrm{deg}(Y) \in \{1,2,4\}$
it remains to consider the case of a cubic hypersurface $Y$ in $\mathbb P_3$. 
The action of $G$ on $Y$ is induced by a four-dimensional (projective) representation of $G$ and the branch curve $C \subset Y$ is the intersection of $Y$ with an invariant quadric (compare proof of Proposition \ref{S5 on degree three}). It follows from the representation and invariant theory of the group $N_{72}$ that there is a unique action of $N_{72}$ on $\mathbb P_3$, a unique invariant cubic hypersurface, namely the Fermat cubic, and a unique invariant quadric hypersurface $\{x_1x_2 + x_3x_4 =0\}$. The necessary computations are carried out in Appendix A.1 in \cite{kf}.

Mukai's $N_{72}$-example is defined by $\{ x_1^3+ x_2 ^3 + x_3^3 +x_4^3= x_1x_2 + x_3x_4+ x_5^2 = 0 \} \subset \mathbb P_4$. An anti-symplectic involution centralizing the action of $N_{72}$ is given by the map $ x_5 \mapsto -x_5$. The quotient of Mukai's example by this involution is the Fermat cubic and the fixed point set of the involution is given by $\{x_1x_2 + x_3x_4= 0 \}$.
\end{proof}
%
\subsection{The group \texorpdfstring{$M_9 =C_3^2 \rtimes Q_8$}{M9}}
Let $G = M_9$ and let $X$ be a K3-surface with a symplectic $G$-action centralized by the antisymplectic involution $\sigma$ such that $\mathrm{Fix}_X(\sigma) \neq \emptyset$.
We proceed in analogy to the case $G=N_{72}$ above. Arguing precisely as in the proof of Lemma \ref{N72notconicbundle} one shows that any $G$-minimal model of $Y$ is a Del Pezzo surface. As a next step we exclude rational branch curves.
\begin{lemma}\label{subgroupsM9}
There are no rational curves in $\mathrm{Fix}_X(\sigma)$.
\end{lemma}
\begin{proof}
Let $n$ be the total number of rational curves in $\mathrm{Fix}_X(\sigma)$. Assume $n \neq 0$, let $C$ be a rational curve in the image of $\mathrm{Fix}_X(\sigma)$ in $Y$ and let $H <G$ be its stabilizer. The action of $H$ on $C$ is effective. We go through the list of finite groups with an effective action on a rational curve. 
Since $M_9$ is a group of symplectic transformations on a K3-surface, its element have order at most eight.
Clearly, $A_5, \, D_{10}, \, D_{14}, \, D_{16} \nless M_9$. If $S_4 < M_9 = C_3^2 \rtimes Q_8$, then $S_4 \cap C_3^2$ is a normal subgroup of $S_4$ and it is therefore trivial. Now $ S_4 = S_4 / (S_4 \cap C_3^2) <  M_9 / C_3^2 =  Q_8$ yields a contradiction. The same argument can be carried out for $A_4$, $D_8$ and $C_8$.  If $D_{12} < M_9 = C_3^2 \rtimes Q_8$, then either $D_{12} \cap C_3^2 = C_3$ and $C_2 \times C_2 = D_{12} / C_3 < M_9 / C_3^2 =  Q_8$ or $D_{12} \cap C_3^2 = \{ \mathrm{id} \}$ and $D_{12} < Q_8$, both are impossible.
It follows that the subgroups of $M_9$ admitting an effective action on a rational curve have index greater than or equal to twelve. Therefore  $n \geq 12$, contrary to inequality \eqref{atmostten} stating $n \leq 10$.
\end{proof}
\begin{proposition}
 The quotient surface $Y$ is $G$-minimal and isomorphic to $\mathbb P_2$. Up to equivalence, there is a unique action of $G$ on $Y$ and the branch locus of $X \to Y$ is given by
$\{x_1^6 + x_2^6+ x_3^6-10( x_1^3x_2^3 + x_2^3x_3^3+ x_3^3x_1^3 ) =0\}$. In particular, $X$ is equivariantly isomorphic to Mukai's $M_9$-example.
\end{proposition}
\begin{proof}
We first check that $Y$ is $G$-minimal. Again, we proceed as in the proof of Theorem \ref{roughclassi} and Lemma \ref{subgroupsM9} above to see a stabilizer group of a Mori fiber has order at most six. If $Y$ is not $G$-minimal, this implies that the total number of Mori fibers is $ \geq 12$, contradicting $m \leq 9$.
Note that $X \to Y$ is not branched along one or two elliptic curves as this would imply $e(Y) =12$ and contradict the fact that $Y$ is a Del Pezzo surface. 

Let $B$ be the branch curve of $ X \to  Y$ and consider the quotient $Q$ of $B$ by the normal subgroup $N = C_3^2$ in $G$. On $Q$ there is an action of $Q_8$ implying that $Q$ is not rational or elliptic.
 It follows that the genus of $Q$ is at least two and the genus of $B$ is at least ten. Adjunction on the Del Pezzo surface $Y$ now implies $g=10$ and $Y \cong \mathbb P_2$. 

It follows from direct computation involving the generators of $M_9$ (cf.\ Appendix A.2 in \cite{kf}) that,
up to natural equivalence, there is a unique action of $M_9$ on the projective plane. We may therefore consider the explicit action of $M_9$ on $\mathbb P_2$ specified by Mukai (cf.\ \cite{mukai}).
Studying the induced action of $M_9$ on the space of sextic curves
one finds three $M_9$-invariant sextic curves, namely
$\{x_1^6 + x_2^6+ x_3^6-10( x_1^3x_2^3 + x_2^3x_3^3+ x_3^3x_1^3 ) =0\}$, 
which is the example presented by Mukai, and additionally the two curves defined by 
$f_a=x_1^6 + x_2^6+ x_3^6 + (18-3a)(x_1^2x_2^2x_3^2) +2(x_1^3x_2^3 + x_1^3x_3^3+ x_2^3x_3^3)
      + a(x_1^4x_2x_3 + x_1x_2^4x_3+ x_1x_2x_3^4)$, 
where $a$ is a solution of the quadratic equation $a^2-6a+36$.
Depending on the choice of $a$, the polynomial $f_a$ is either invariant or semi-invariant with respect to the action of $M_9$.

We need to prove that $X$ is not the double cover of $\mathbb P_2$ branched along $\{f_a=0\}$. If this was the case consider the fixed point $p=[0:1:-1]$ of the automorphism $I \in Q_8 < M_9 < \mathrm{PSL}_3(\mathbb C)$ given by
\[
\tilde I= \frac{1}{\xi -\xi^2}
 \begin{pmatrix}
1 & 1 & 1 \\
1 & \xi & \xi^2\\
1 & \xi^2& \xi
 \end{pmatrix} \in \mathrm{SL}_3(\mathbb C)
\]
for some third root of unity $\xi$
and note that $f_a(p)=0$. In particular, the fiber $\pi^{-1}(p)$ consists of one point $x \in X$. We linearize the $\langle I \rangle \times \langle \sigma \rangle$-action at $x$. In suitably chosen coordinates the action of the symplectic automorphism $I$ of order four is of the form  $(z,w) \mapsto (iz, -iw)$. Since the action of $\sigma$ commutes with $I$, the $\sigma$-quotient of $X$ is locally given by
$(z,w) \mapsto (z^2, w)$ or $(z,w) \mapsto (z, w^2)$.
It follows that the action of $I$ on $Y$ is locally given by either
$ (x,y) \mapsto (-x, -i y)$ or
$ (x, y) \mapsto (i x, -y)$.
In particular, the local linearization of $I$ at $p$ has determinant $\neq 1$. 
By a direct computation using the explicit form $\tilde I \in \mathrm{SL}_3(\mathbb C)$ of the generator $I$, in particular the facts that $\mathrm{det}(\tilde I) =1$ and $\tilde I v =v$ for $[v]=p$, we obtain a contradiction.
 This completes the proof of the proposition.
\end{proof}
\begin{remark}\label{M9 symplectic}
In the proof of the proposition above we have observed that an element of $\mathrm{SL}_3(\mathbb C)$ does not necessarily lift to a symplectic transformation on the double cover of $\mathbb P_2$ branched along a sextic given by an invariant polynomial.
Mukai's $M_9$-example $X$ is the double cover of $\mathbb P_2$ branched along the sextic curve $\{x_1^6 + x_2^6+ x_3^6-10( x_1^3x_2^3 + x_2^3x_3^3+ x_3^3x_1^3 ) =0\}$ and for this particular example, the action of $M_9$ does lift to a group of symplectic transformation as stated by Mukai.\\
To see this consider the set $\{a,b,I,J\}$ of generators of $M_9 =C_3^2 \rtimes Q_8= (\langle a \rangle  \times \langle b \rangle  ) \rtimes Q_8$. Since $a$ and $b$ are commutators in $M_9$, they can be lifted to symplectic transformation $\overline a, \overline b$ on $X$. For $I,J$ consider the linearization at the fixed point $[0:1:-1] \in \mathbb P_2$ and, using the explicit realization of the group, one checks that it has determinant one. Since $[0:1:-1]$ is \emph{not} contained in the branch set of the covering, its preimage in $X$ consists of two points $p_1,p_2$. We can lift $I$ ($J$, respectively) to a transformation of $X$ fixing both $p_1,p_2$ and a neighbourhood of $p_1$ is $I$-equivariantly isomorphic to a neighbourhood of $ [0:1:-1] \in \mathbb P_2$. In particular, the action of the lifted element $\overline I$ ($\overline J$, respectively) is symplectic. On $X$ there is the action of a degree two central extension $E$ of $M_9$.
The elements $\overline a, \overline b, \overline I, \overline J$ generate a subgroup $\tilde M_9$ of $E_\mathrm{symp}$ mapping onto $M_9$. Since $E_\mathrm{symp} \neq E$, the order of $\tilde M_9$ is 72 and it follows that $\tilde M_9$ is isomorphic to $M_9$. In particular $E$ splits as $E_\mathrm{symp} \times C_2$ with  $E_\mathrm{symp}= M_9$.
\end{remark}
%
\subsection{The group \texorpdfstring{$T_{48} = Q_8 \rtimes S_3$}{T48}}
We begin by specifying the group structure of $T_{48}$: the element $c$ of order three in $S_3$ acts on $Q_8$ by permuting $I,J,K$ and an element $d$ of order two acts by exchanging $I$ and $J$ and mapping $K$ to $-K$.
Now let $X$ be a K3-surface with an action of $T_{48} \times C_2$ where the action of $G = T_{48}$ is symplectic and the generator $\sigma$ of $C_2$ is antisymplectic and has fixed points. 
\begin{lemma}\label{not two}
A $G$-minimal model $Y_\mathrm{min}$ of $Y$ is either $\mathbb P_2$, a Hirzebruch surface $\Sigma_k$ with $k >2$, or $e(Y_\mathrm{min}) \geq 9$.
\end{lemma}
\begin{proof}
Let us first consider the case where $ Y_\mathrm{min}$ is a Del Pezzo surface.

If $Y_\mathrm{min} \cong \mathbb P_1 \times \mathbb P_1$, then both canonical projections are equivariant with respect to the index two subgroup $G'= Q_8 \rtimes C_3$. 
Since $Q_8$ has no effective action on $\mathbb P_1$, it follows that the subgroup $Z =\{+1,-1\} < Q_8$ acts trivially on the base. This holds with respect to both projections and the subgroup $Z$ acts trivially on $Y_\mathrm{min}$, a contradiction. 

Using the group structure of $T_{48}$ one checks that the only nontrivial normal subgroup $N$ of $T_{48}$ such that $N \cap Q_8 \neq Q_8$ is the center $Z= \{+1,-1\}$ of $T_{48}$. It follows that $T_{48}$ is neither a subgroup of $(\mathbb C^*)^2 \rtimes (S_3 \times C_2)$ nor a subgroup of any of the automorphism groups $C_2^4 \rtimes \Gamma$ for $\Gamma \in \{C_2, C_4, S_3, D_{10}\}$ of a Del Pezzo surface of degree four. 
Furthermore,
$T_ {48} \nless S_5$ and consequently $\mathrm{deg}(Y_\mathrm{min}) \neq 4,5,6$.

So if $Y_\mathrm{min}$ is a Del Pezzo surface, then $Y_\mathrm{min} \cong \mathbb P_2$ or $e(Y_\mathrm{min}) \geq 9$ .

Let us now turn to the case where $Y_\mathrm{min}\to \mathbb P_1$ is an equivariant conic bundle.
The center $Z = \{+1,-1\}$ of $G= T_{48}$ acts trivially on the base and has two fixed points in the generic fiber. Let $\Lambda_1$ and $\Lambda_2$ denote the two curves of $Z$-fixed points in $Y_\mathrm{min}$. 

We first show that $Y_\mathrm{min}$ is not a conic bundle with singular fibers.
Any singular fiber $F$ is the union of two (-1)-curves $F_1,F_2$ meeting transversally in one point. We consider the action of $Z$ on this union of curves. The group $Z$ does not act trivially on either component of $F$ since linearization at a smooth point of $F$ would yield a trivial action of $Z$ on $Y_\mathrm{min}$.
Consequently,
it has either one or three fixed points on $F$. 
The first is impossible since $\Lambda_1$ and $\Lambda_2$ intersect $F$ in two points. It follows that $Z$ stabilizes each curve $F_i$. We linearize the action of $Z$ at the point of intersection $F_1 \cap F_2$. The intersection is transversal and the action of $Z$ is by $-\mathrm{Id}$ on the tangent space $T_{F_1} \oplus T_{F_2}$ contradicting the fact the $Z$ acts trivially on the base. Thus $Y_\mathrm{min}$ is not a conic bundle with singular fibers.

It remains to consider the case where $p:Y_\mathrm{min} = \Sigma _k \to \mathbb P_1$ is a Hirzebruch surface.
The curves $\Lambda_1$ and $\Lambda_2$ are disjoint sections of $p$. 
This is only possible if $\Lambda_1^2 = -\Lambda_2^2 = k$. 
In particular, the action of $T_{48}$ on $\Sigma_k$ stabilizes each curve $\Lambda_i$.
Removing the exceptional section $\Lambda_2= E_\infty$ from $\Sigma_k$, we obtain $H^k$, the $k^\text{th}$ power of the hyperplane bundle of $\mathbb P_1$. We chose the section $\Lambda_2$ to be the zero section and conclude that the action of $T_{48}$ on $H^k$ is by bundle automorphisms. 
If $k=2$, then $H^k$ is the anticanonical line bundle of $\mathbb P_1$ and the effective action of $S_4$ on the base induces an action of $S_4$ on $H^2$ by bundle automorphisms. It follows that $T_{48}$ splits as $S_4 \times C_2$, a contradiction. Thus, if $Y_\mathrm{min}$ is a Hirzebruch surface $\Sigma_k$, then $k \neq 2$. 
\end{proof}
\begin{lemma}\label{no rat T48}
 There are no rational curves in $\mathrm{Fix}_X(\sigma)$.
\end{lemma}
\begin{proof}
We let $n$ denote the total number of rational curves in $\mathrm{Fix}_X(\sigma)$ and assume $n >0$. Recall $n \leq 10$, let $C$ be a rational curve in $B=\pi(\mathrm{Fix}_X(\sigma)) \subset Y$ and let $H = \mathrm{Stab}_G(C) < G$ be its stabilizer group. The action of $H$ on $C$ is effective, the index of $H$ in $G$ is at most 8. Using the quotient homomorphism $T_{48} \to T_{48}/Q_8 = S_3$ one checks that $T_{48}$ does not contain $O_{24}=S_4$ or $T_{12}= A_4$ as a subgroup. It follows that $H$ is a cyclic or a dihedral group.

If $H\in \{C_6, C_8, D_8\}$, then $H$ and all conjugates of $H$ in $G$ contain the center $Z= \{+1,-1\}$ of $G$. It follows that $Z$ has two fixed point on each curve $gC$ for $g \in G$. Since there are six (or eight) distinct curves $gC$ in $Y$, it follows that $Z$ has at least 12 fixed points in $Y$ and in $X$. This contradicts to assumption that $Z < G$ acts symplectically on $X$ and therefore has eight fixed points in the K3-surface $X$.

It remains to study the cases $H = D_{12}$, i.e., $n=4$, and $H = D_6$, i.e., $n = 8$. 

We note that a Hirzebruch surface has precisely one curve with negative self-intersection and only fibers have self-intersection zero. A Del Pezzo surface does not contains curves of self-intersection less than $-1$. The rational branch curves must therefore meet the union of Mori fibers in $Y$.

The total number of Mori fibers is bounded by $n+9$. We study the possible stabilizer subgroups $\mathrm{Stab}_G(E) < G$ of Mori fibers. A Mori fiber $E$ with self-intersection (-1) meets the branch locus $B$ in one or two points and its stabilizer is either cyclic or dihedral. If $\mathrm{Stab}_G(E) \in \{C_4, D_8\}$, then the points of intersection of $E$ and $B$ are fixed points of the center $Z$ of $G$ and we find too many $Z$-fixed points on $X$. 

Assume $n=4$ and let $R_1, \dots R_4$ be the rational curves in $B$. We denote by $\tilde R_i$ their images in $Y_\mathrm{min}$. The total number $m$ of Mori fibers is bounded by 12.
We go through the list of possible configurations:
\begin{itemize}
 \item 
If $m = 4$, there is no invariant configuration of Mori fibers such that the contraction maps the four rational branch curves to an admissible configuration on the Hirzebruch or Del Pezzo surface $Y_\mathrm{min}$.
\item
If $m= 6$, then $\mathrm{Stab}_G(E) = C_8$ and the points of intersection of $E$ and $B$ are $Z$-fixed. Since $Z$ has at most eight fixed points on $B$, it follows that each curve $E$ meets $B$ only once. The images $\tilde R_i$ of the $R_i$ contradict our observations about curves in Del Pezzo and Hirzebruch surfaces.
\item
If $m=8$ and all Mori fibers have self-intersection $-1$, then each Mori fiber meets $\bigcup R_i$ in a $Z$-fixed point.
 Since there at at most eight such points, it follows that each Mori fibers meets $\bigcup R_i$ only once and their contractions does not transform the curves $R_i$ sufficiently. 
\item
If $m=8$ and only four Mori fibers have self-intersection $-1$, we consider the four Mori fibers of the second reduction step. Each of these meets a Mori fiber $E$ of the first step in precisely one point. By invariance, this would have to be a fixed points of the stabilizer $ \mathrm{Stab}_G(E)= D_{12}$, a contradiction.
\item
If $m=12$, then either $e(Y_\mathrm{min}) = 3$ and there exist a branch curve $B_{g=2}$ of genus two or $e(Y_\mathrm{min}) = 4$ and $B = \bigcup R_i$. In the first case, $Y_\mathrm{min}  \cong \mathbb P_2$ and twelve Mori fibers are not sufficient to transform $B = B_{g=2} \cup \bigcup R_i$ into an admissible configuration of curves in the projective plane.
So $Y_\mathrm{min} = \Sigma_k$ for $k > 2$.
Recall that $Z$ has two fixed points in each fiber of $p: \Sigma_k \to \mathbb P_1$, i.e., the $Z$-action on $\Sigma_k$ has two disjoint curves of fixed points. As was remarked above, these curves are the exceptional section $E_\infty$ of self-intersection $-k$ and a section $E_0 \sim E_\infty + k F$ of self-intersection $k$. Here $F$ denotes a fiber of $p: \Sigma_k \to \mathbb P_1$. There is no automorphisms of $\Sigma_k$ mapping $E_\infty$ to $E_0$.
Each rational branch curve $\tilde R_i$ has two $Z$-fixed points. These are exchanged by an element of $\mathrm{Stab}_G(R_i)$ and therefore both lie on either $E_\infty$ or $E_0$, i.e., $\tilde R_i$ cannot have nontrivial intersection with both $E_0$ and $E_\infty$. By invariance all curves $\tilde R_i$ either meet $E_0$ or $E_\infty$ and not both.
Using the fact that $\sum \tilde R_i$ is linearly equivalent to $-2K_{\Sigma_k} \sim 4 E_\infty +(2k +4)F$ we find that $\tilde R_i \cdot E_\infty = 0$ and $k=2$, a contradiction to Lemma \ref{not two}. 
\end{itemize}
We have shown that
all possible configurations in the case $n \neq 4$ lead to a contradiction. We now turn to the case $n=8$ and let $R_1, \dots R_8$ be the rational branch curves. The total number of Mori fibers is bounded by 16. Note that by invariance, the $G$-orbit of a Mori fiber meets $\bigcup R_i$ in at least 16 points or not at all. In particular, Mori fibers meeting $R_i$ come in orbits of length $\geq 8$. As above, we go through the list of possible configurations.
\begin{itemize}
 \item 
If $m =16$, then the set of all Mori fibers consists of one orbit of length 16 or of two orbits of length eight. If all 16 Mori fibers meet $B$, then each meets $B$ in one point and $R_i$ is mapped to a (-2)-curve in $Y_\mathrm{min}$. If only eight Mori fibers meet $B$, then each of the eight Mori fibers of the second reduction step meets one Mori fiber $E$ of the first reduction step in one point. This point has to be a $\mathrm{Stab}_G(E)$-fixed point. The $\mathrm{Stab}_G(E)$-fixed points on $E$ must however coincide with the points $E \cap B$, a contradiction.
\item
If $m=12$, then the set of all Mori fibers consists of a single $G$-orbit and each curve $R_i$ meets three distinct Mori fibers. Their contraction transforms $R_i$ into a (-1)-curve on $Y_\mathrm{min}$. It follows that $Y_\mathrm{min}$ contains at least eight (-1)-curves and is a Del Pezzo surface of degree $\leq 5$. We have seen above that $\mathrm{deg}( Y_\mathrm{min}) \neq 4,5$ and therefore $e(Y_\mathrm{min}) \geq 9$. With $m=12$ and $n=8$, this contradicts the Euler characteristic formula $24 = 2 e(Y_\mathrm{min}) +2m -2n +(2g-2)$.
\item
If $m=8$ there is no invariant configuration of Mori fibers such that the contraction maps the eight rational branch curves to an admissible configuration on the Hirzebruch or Del Pezzo surface $Y_\mathrm{min}$.
\end{itemize}
This completes the proof of the lemma.
\end{proof}
Since there is an effective action of $T_{48}$ on $\mathrm{Fix}_X(\sigma)$, it is neither an elliptic curve nor the union of two elliptic curves. It follows that $ X \to Y$ is branched along a single $T_{48}$-invariant curve $B$ with $g(B) \geq 2$.
\begin{lemma}
The genus of $B$ is neither three nor four.
\end{lemma}
\begin{proof}
We consider the quotient $Q = B/Z$ of the curve $B$ by the center $Z$ of $G$ and apply the Euler characteristic formula,  $e(B) = 2 e(Q) - |\mathrm{Fix}_B(Z)|$. On $Q$ there is an effective action of the group $G/Z = (C_2 \times C_2) \rtimes C_3 = S_4$. Using Remark \ref{S4 not on g=1,2} we see that $e(Q) \in \{2, -4, -6,-8 \dots\}$.

If $g(B)=3$, then $e(B) = -4$ and the only possibility is $Q \cong \mathbb P_1$ and $|\mathrm{Fix}_B(Z)| =8$. In particular, all $Z$-fixed points on $X$ are contained in the curve $R = \pi^{-1}(B)$. Let $A < G$ be the group generated by $I \in Q_8 = \{ \pm 1, \pm I, \pm J, \pm K \}$. The four fixed points of $A$ in $X$ are contained in $\mathrm{Fix}_X(Z) = \mathrm{Fix}_B(Z)$ and the quotient group $A/Z \cong C_2$ has four fixed points in $Q$. This is a contradiction.

If $g(B)=4$, then $e(B) = -6$ and the only possibility is $Q \cong \mathbb P_1$ and $|\mathrm{Fix}_B(Z)| =10$. This contradicts the fact that $Z$ has at most eight fixed points in $B$ since it has precisely eight fixed points in $X$.
\end{proof}
In Lemma \ref{not two} we have reduced the classification to the cases $e(Y_\mathrm{min}) \in \{3,4, 9, 10, 11\}$. In the following, we will exclude the cases $e(Y_\mathrm{min}) \in \{4, 9, 10\}$ and describe the remaining cases more precisely. Recall that the maximal possible stabilizer subgroup of a Mori fiber is $D_{12}$, in particular, $m = 0$ or $ m \geq 4$.
\begin{lemma}
 If  $e(Y_\mathrm{min}) =3$, then $Y_\mathrm{min} = Y = \mathbb P_2$ and $X \to Y$ is branched along the curve $\{x_1x_2(x_1^4-x_2^4) + x_3^6=0\}$. In particular, $Y$ is equivariantly isomorphic to Mukai's $T_{48}$-example.
\end{lemma}
\begin{proof}
Let $ M: Y \to \mathbb P_2$ denote a Mori reduction of $Y$ and let $B \subset Y$ be the branch curve of the covering $X \to Y$. 
If $Y = Y_\mathrm{min}$, then $B= M(B)$ is a smooth sextic curve.
If $Y \neq Y_\mathrm{min}$, then the Euler characteristic formula with $m \in \{4,6,8\}$ shows that $g(B) \in \{2,4,6\}$. The case $m=6$, $g(B) =4$ has been excluded by the previous lemma.

If $m=4$, then the stabilizer group of each Mori fiber is $D_{12}$ and each Mori fiber meets $B$ in two points. 
The image $M(B)$ of $B$ in $Y_\mathrm{min}$ is an irreducible singular sextic.

If $m=8$, then $g(B) =2$ and $B^2 =4$. Since the self-intersection number $M(B)^2$ must be a square, one checks that all possible invariant configurations of Mori fibers yield $M(B)^2 =36$ and involve Mori fibers meeting $B$ in two points. In particular, also in this case $M(B)$ must be a singular sextic.

By explicit computation using the group structure of $T_{48}$ one can determine the unique action of $T_{48}$ on $\mathbb P_2$. This is carried out in detail in Section 4.10 in \cite{kf}.
Using the explicit form of the $T_{48}$-action and the fact that the commutator subgroup of $T_{48}$ is $Q_8 \rtimes C_3$ one then checks that any invariant curve of degree six is of the form 
\[
C_\lambda = \{ x_1x_2(x_1^4-x_2^4) + \lambda x_3^6 =0\}
\]
for some $\lambda \in \mathbb C^*$.
In order to avoid this calculation, one can also argue that the polynomial $x_1x_2(x_1^4-x_2^4)$ is the lowest order invariant of the octahedral group $S_4 \cong T_{48} / Z$. 
The curve $C_\lambda$ is smooth and it follows that $Y = Y_\mathrm{min}$.
We may adjust the coordinates equivariantly such that $\lambda =1$ and find that our surface $X$ is precisely Mukai's $T_{48}$-example.
\end{proof}
\begin{remark}\label{T48 symplectic}
As stated by Mukai, the action of $T_{48}$ on $\mathbb P_2$ does indeed lift to a symplectic action of $T_{48}$ on the double cover of $\mathbb P_2$ branched along the invariant curve $\{ x_1x_2(x_1^4-x_2^4) +  x_3^6 =0\}$. The elements of the commutator subgroup can be lifted to symplectic transformation on the double cover $X$. 
The remaining generator $d$ 
is an involution fixing the point $[0:0:1]$. Any involution $\tau$ with a fixed point $p$ outside the branch locus can be lifted to a symplectic involution on the double cover $X$ as follows:
The linearized action of $\tau$ at $p$ has determinant $\pm 1$. We consider the lifting $\tilde \tau$ of $\tau$ fixing both points in the preimage of $p$. Its linearization coincides with the linearization on the base and therefore also has determinant $\pm 1$. 
In particular, $\tilde \tau$ is an involution. 
It follows that either $\tilde \tau$ or the second choice of a lifting $\sigma \tilde \tau$ acts symplectically on $X$. 
The group generated by all lifted automorphisms is either isomorphic to $T_{48}$ or to its degree two central extension $E$
acting on the double cover.
Since $E_\mathrm{symp} \neq E$ the later is impossible it follows that $E$ splits as $E_\mathrm{symp} \times C_2$ with $E_\mathrm{symp} = T_{48}$.
\end{remark}
Finally, we return to the remaining possibilities $e(Y_\mathrm{min}) \in \{4, 9, 10, 11\}$.
\begin{lemma}
$e(Y_\mathrm{min}) \not\in \{4,9,10\}$.
\end{lemma}
\begin{proof}
Recalling that the genus of the branch curve $B$ is neither three nor four and that $m$ is either zero or $\geq 4$, we may exclude $e(Y_\mathrm{min}) =9,10$ using the Euler characteristic formula $12 = e(Y_\mathrm{min}) +m +g-1$. It remains to consider the case $Y_\mathrm{min} = \Sigma _k$ with $k >2$ and we claim that this is impossible.

Let $M: Y \to Y_\mathrm{min} = \Sigma_k$ denote a (possibly trivial) Mori reduction of $Y$. The image $M(B)$ of $B$ in $\Sigma _k$ is linearly equivalent to $-2K_{\Sigma _k}$. Now $M(B) \cdot E_\infty = 2(2-k) < 0$ and it follows that $M(B)$ contains the rational curve $E_\infty$. This is a contradiction since $B$ does not contain any rational curves by Lemma \ref{no rat T48}.
\end{proof}
In the last remaining case, i.e., $e(Y_\mathrm{min})=11$, the quotient surface $Y$ is a $G$-minimal Del Pezzo surface of degree 1. Consulting \cite{dolgachev}, Table 10.5, we find that $Y$ is a hypersurface in weighted projective space $\mathbb P(1,1,2,3)$ defined by the weighted homogeneous equation
$x_1x_2(x_1^4-x_2^4)+ x_3^3+x_4^2$.
This follows from the invariant theory of the group $S_4 \cong T_{48}/Z$ and fact that the anticanonical map realizes  $Y$ as a double cover of a quadric cone $Q$ in $\mathbb P_3$ branched along the intersection of $Q$ with a cubic hypersurface.

The linear system of the anticanonical divisor $K_Y$ has precisely one base point $p$. In coordinates $[x_1:x_2:x_3:x_4]$ this point is given as $[0:0:1:i]$. It is fixed by the action of $T_{48}$. The linearization of $T_{48}$ at $p$ is given by the unique faithful 2-dimensional represention of $T_{48}$. 
It follows that there is a unique action of $T_{48}$ on $Y$. The branch curve $B$ is linearly equivalent to $-2K_Y$, i.e., $B = \{s=0\}$ for a section $s \in \Gamma(Y, \mathcal O(-2K_Y))$ which is either invariant or semi-invariant.

The adjunction formula for hypersurfaces in weighted projective space (cf.\ Theorem 3.3.4 in \cite{DolgWPS}) yields $\mathcal O (-2K_Y)) = \mathcal O _Y(2)$. The four-dimensional space of sections $\Gamma(Y, \mathcal O(-2K_Y))$ is generated by the weighted homogeneous polynomials $x_1^2, x_2^2, x_1x_2, x_3$. 
We consider the map $ Y \to \mathbb P (\Gamma(Y, \mathcal O(-2K_Y))^*)$ associated to $|-2K_Y|$. Since this map is equivariant with respect to $\mathrm{Aut}(Y)$, the fixed point $p$ is mapped to a fixed point in $\mathbb P (\Gamma(Y, \mathcal O(-2K_Y))^*)$. It follows that the section corresponding to the homogeneous polynomial $x_3$ is invariant or semi-invariant with respect to $T_{48}$. It is the only section of $\mathcal O(-2K_Y)$ with this property since the representation of $T_{48}$ on the span of $x_1^2, x_2^2, x_1x_2$ is irreducible.
It follows that $B = \{x_3 =0\}$.

In order to prove existence, let $b: Y \to \mathbb P_2$ be the blow down of eight disjoint (-1)-curves in $Y$. The curve $C = b(B)$ is a singular sextic. The double cover $X$ of $Y$ branched along $B$ is the minimal desingularization of the double cover $X_\mathrm{sing}$ of $\mathbb P_2$ branched along $C$. In particular, one finds that $X$ is a K3-surface. 
It remains to check that the action of $T_{48}$ on $Y$ lifts to a group of symplectic transformation on $X$. First note that $B$ does not contain the base point $p$. 
For $I,J,K, c \in T_{48}$ we choose liftings $ \overline I, \overline J, \overline K , \overline c \in \mathrm{Aut}(X)$ fixing both points in $\pi^{-1}(p)= \{p_1,p_2\}$. The linearization of $ \overline I, \overline J, \overline K , \overline c$ at $p_1$ is the same as the linearization at $p$ and in particular has determinant one. 
By the general considerations in Remark \ref{T48 symplectic} the involution $d$ can be lifted to a symplectic involution on $X$.
The symplectic liftings of $I,J,K,c,d$ generate a subgroup $\tilde G$ of $\mathrm{Aut}(X)$ which is isomorphic to either $T_{48}$ or to the central degree two extension of $T_{48}$ acting on $X$. 
In analogy to Remarks \ref{M9 symplectic} and \ref{T48 symplectic} we conclude that $\tilde G \cong T_{48}$ and the action of $T_{48}$ on $Y$ induces a symplectic action of $T_{48}$ on the double cover $X$.

This completes the classification of K3-surfaces with $T_{48} \times C_2$-symmetry. We have shown:
\begin{proposition}
 Let $X$ be a K3-surface with a symplectic action of the group $T_{48}$ centralized by an antisymplectic involution $\sigma$ with $\mathrm{Fix}_X(\sigma) \neq \emptyset$. Then $X$ is equivariantly isomorphic either to Mukai's $T_{48}$-example or to the double cover of 
$\{x_1x_2(x_1^4-x_2^4)+ x_3^3+x_4^2=0\} \subset \mathbb P(1,1,2,3)$
branched along $\{x_3=0\}$
\end{proposition}
\begin{remark}
The automorphism group of the Del Pezzo surface $Y = \{x_1x_2(x_1^4-x_2^4)+ x_3^3+x_4^2=0\} \subset \mathbb P(1,1,2,3)$ is the trivial central extension $C_3 \times T_{48}$.  By contruction, the curve $B=\{s=0\}$ is invariant with respect to the full automorphism group. The double cover $X$ of $Y$ branched along $B$ carries the action of a finite group $\tilde G$ of order $2 \cdot 3 \cdot 48 = 288$ containing $T_{48} < \tilde G_\mathrm{symp}$. Since $T_{48}$ is a maximal group of symplectic transformations, we find $T_{48} = \tilde G_\mathrm{symp}$.
In analogy to the proof of Claim 2.1 in \cite{OZ168}, one can check that 288 is the maximal order of a finite group $H$ acting on a K3-surface with $T_{48} < H_\mathrm{symp}$. It follows that $\tilde G$ is a maximal finite subgroup of $\mathrm{Aut}(X)$.
\end{remark}
\subsection*{Conclusion}
The classification and non-existence results for the individual groups from Table \ref{TableMukai} obtained in the preceding subsections yield the proof of Theorem \ref{thm mukai times invol}.
%
%
%
\providecommand{\bysame}{\leavevmode\hbox to3em{\hrulefill}\thinspace}

\end{document}